\documentclass[final,smallcondensed,numbook]{svjour3}

\usepackage{graphicx} 
\usepackage{amssymb,amsmath,amsbsy,mathrsfs,amsfonts}
\usepackage[english]{babel}					\usepackage[dvipsnames]{xcolor}
\usepackage[utf8]{inputenc}
\usepackage{multirow} 
\usepackage[parfill]{parskip} 
\usepackage{cancel,mathtools,empheq,latexsym}
\usepackage{booktabs,multicol,multirow,tabularx,array} 
\usepackage{bm}
\usepackage{mathrsfs}
\usepackage{comment}
\usepackage[font={small,it}]{caption} %
\usepackage{hyperref}
\hypersetup{
    colorlinks=true,       
    linkcolor=red,          
    citecolor=blue,        
    filecolor=magenta,      
    urlcolor=cyan           
}

\newcommand{\bs}[1]{\boldsymbol{#1}}
\newcommand{\vertiii}[1]{{\left\vert\kern-0.25ex\left\vert\kern-0.25ex\left\vert #1 
    \right\vert\kern-0.25ex\right\vert\kern-0.25ex\right\vert}}    
\newcommand{\uline}{\underline}
\newcommand{\dis}{\displaystyle}
\newcommand{\midiv}{\nabla\hspace{-0.05cm}\cdot }

\newcommand{\C}{\mathbb{C}}
\newcommand{\nablat}{\nabla^{\top}}
\newcommand{\re}[1]{\mathrm{Re}\!\left(#1\right)}

\newcommand{\bos}[1]{\boldsymbol{#1}}
\newcommand{\ten}[1]{\bos{#1}}
\newcommand{\mcal}[1]{\mathcal{#1}}
\newcommand{\ul}[1]{\underline{#1}}
\newcommand{\conj}[1]{\overline{#1}}
\newcommand{\lef}{\left}
\newcommand{\rig}{\right}
\newcommand{\ra}[2]{\lef. #1\rig|_{#2}} 

\newcommand{\hooke}{\mathbf{C}}
\newcommand{\hookei}{\mathbf{C}^{-1}}
\newcommand{\strain}[1]{\ten{\varepsilon}\lef(#1\rig)}
\newcommand{\rA}{\rho_f}
\newcommand{\rE}{\rho_E}
\newcommand{\vinc}{v^\text{inc}}
\newcommand{\vtot}{v^\text{tot}}
\newcommand{\vh}{v_h}
\newcommand{\vhg}{\widehat{v}_h}

\newcommand{\uv}{\bos{u}}
\newcommand{\uvh}{\uv_h}
\newcommand{\uvhg}{\widehat{\uv}_h}

\newcommand{\si}{\ten{\sigma}}
\newcommand{\sih}{\si_h}
\newcommand{\sihg}{\widehat{\si}_h}

\newcommand{\q}{\bos{q}}
\newcommand{\qh}{\q_h}
\newcommand{\qhg}{\widehat{\q}_h}

\newcommand{\g}{\ten{\gamma}}
\newcommand{\gh}{\g_h}

\newcommand{\etat}{\ten{\eta}}

\newcommand{\taut}{\ten{\tau}}

\newcommand{\tv}{\bos{t}}

\newcommand{\muv}{\bos{\mu}}
\newcommand{\rv}{\bos{r}}

\newcommand{\efe}{\bos{f}}

\newcommand{\te}{\mcal{T}_E}
\newcommand{\dte}{\partial\te}
\newcommand{\ta}{\mcal{T}_A}
\newcommand{\dta}{\partial\ta}
\newcommand{\ee}{\mcal{E}_E}
\newcommand{\ea}{\mcal{E}_A}
\newcommand{\oma}{\Omega_A}
\newcommand{\ome}{\Omega_E}
\newcommand{\ga}{\Gamma}
\newcommand{\nor}{\bos{n}}
\newcommand{\nore}{\bos{n}_E}
\newcommand{\nora}{\bos{n}_A}
\newcommand{\gaa}{\Gamma_A}
\newcommand{\gaadir}{\gaa^D}
\newcommand{\gaaneu}{\gaa^N}

\newcommand{\pk}[2]{\mcal{P}_{#1}\!\lef(#2\rig)}
\newcommand{\pkv}[2]{\bos{\mcal{P}}_{#1}\!\lef(#2\rig)}
\newcommand{\pkt}[2]{\underline{\ten{\mcal{P}}}_{#1}\!\lef(#2\rig)}
\newcommand{\pkttil}[2]{\ul{\ten{\widetilde{\mcal{P}}}}_{#1}\!\lef(#2\rig)}
\newcommand{\aspace}[1]{\uline{\bos{A}}\!\lef(#1\rig)}
\newcommand{\atilspace}[1]{\uline{\bos{\widetilde{A}}}\!\lef(#1\rig)}
\newcommand{\asspace}[1]{\uline{\bos{AS}}\!\lef(#1\rig)}
\newcommand{\bspace}[1]{\uline{\bos{B}}\!\lef(#1\rig)}
\newcommand{\vtspace}[1]{\uline{\bos{V}}\!\lef(#1\rig)}
\newcommand{\ld}[1]{L^2\!\lef(#1\rig)}
\newcommand{\ldv}[1]{\bos{L}^2\!\lef(#1\rig)}
\newcommand{\ldt}[1]{\uline{\bos{L}}^2\!\lef(#1\rig)}
\newcommand{\vhspace}{\uline{\bos{V}}_h}
\newcommand{\whvespace}{\bos{W}_h^E}
\newcommand{\whvaspace}{\bos{W}_h^A}
\newcommand{\ahspace}{\uline{\bos{A}}_h}
\newcommand{\ahzspace}{\uline{\bos{A}}^0_h}
\newcommand{\ahcspace}{\uline{\bos{A}}^c_h}
\newcommand{\mhvspace}{\bos{M}_h}
\newcommand{\whspace}{W_h}
\newcommand{\mhspace}{M_h}

\newcommand{\ipta}[2]{( #1,#2 )_{\ta}}
\newcommand{\ipte}[2]{( #1,#2 )_{\te}}
\newcommand{\ipba}[2]{\langle #1,#2 \rangle_{\dta}}
\newcommand{\ipbe}[2]{\langle #1,#2 \rangle_{\dte}}
\newcommand{\ipfe}[2]{\langle #1,#2 \rangle_{\dte\setminus\ga}}
\newcommand{\ipk}[2]{( #1,#2 )_{K}}

\newcommand{\ipg}[2]{\langle #1,#2 \rangle_{\ga}}
\newcommand{\nte}[1]{\lef\| #1 \rig\|_{\te}}
\newcommand{\ntehookei}[1]{\lef\| #1 \rig\|_{\te,\hookei}}
\newcommand{\nta}[1]{\lef\| #1 \rig\|_{\ta}}
\newcommand{\nbta}[1]{\lef\| #1 \rig\|_{\dta}}
\newcommand{\nbte}[1]{\lef\| #1 \rig\|_{\dte}}

\newcommand{\esi}{\ten{e_{\bos{\sigma}}}}
\newcommand{\eu}{\bos{e_{\bos{u}}}}
\newcommand{\eg}{\ten{e_{\bos{\gamma}}}}
\newcommand{\eug}{\bos{e_{\bos{\widehat{u}}}}}
\newcommand{\eq}{\bos{e_{\bos{q}}}}
\newcommand{\ev}{e_v}
\newcommand{\evg}{e_{\widehat{v}}}
\newcommand{\esig}{\ten{e_{\bos{\widehat{\sigma}}}}}
\newcommand{\eqg}{\bos{e_{\bos{\widehat{q}}}}}
\newcommand{\piv}{\ten{\underline{\Pi}_E}}
\newcommand{\piwe}{\bos{\Pi_E}}
\newcommand{\piwa}{\bos{\Pi_A}}
\newcommand{\pia}{\ten{\underline{\Pi}}}
\newcommand{\piw}{\Pi_A}
\newcommand{\pemv}{\bos{P_M}}
\newcommand{\pem}{P_M}
\newcommand{\delsi}{\ten{\delta_{\sigma}}}
\newcommand{\delg}{\ten{\delta_{\gamma}}}
\newcommand{\delu}{\bos{\delta_{u}}}
\newcommand{\delq}{\bos{\delta_{q}}}
\newcommand{\delv}{\delta_{v}}

\definecolor{UAred}{RGB}{171, 5, 32}

\title{A coupled HDG discretization for the interaction between acoustic and elastic waves}
\titlerunning{Coupled HDG for acoustic/elastic waves}

\author{Fernando Artaza-Covarrubias
    \thanks{Partially funded by ANID-Chile through the grant Fondecyt Regular 1240183.} 
    \and Tonatiuh S\'anchez-Vizuet
    \thanks{Partially funded by the U. S. National Science Foundation through the grant NSF-DMS-2137305.} 
    \and Manuel Solano
    \thanks{Partially funded by ANID-Chile through the grants  Fondecyt Regular 1240183 and Basal FB210005.} 
    }
\institute{Fernando Artaza-Covarrubias
            \at Department of Mathematics, Simon Fraser University, Canada.
            \email{faa61@sfu.ca} \and
           Tonatiuh S\'anchez-Vizuet
            \at Department of Mathematics, The University of Arizona, USA.
            \email{tonatiuh@arizona.edu} \and
           Manuel Solano
            \at  Department of Mathematical Engineering, Faculty of Mathematical and Physical Sciences and Center for Research in Mathematical Engineering CI$^2$MA  Universidad de Concepci\'on, Chile.
            \email{msolano@ing-mat.udec.cl}
    }

\authorrunning{Artaza-Covarrubias, S\'anchez-Vizuet, Solano}

\date{}

\begin{document}

\maketitle

\begin{abstract}
We propose and analyze an HDG scheme for the Laplace-domain interaction between a transient acoustic wave and a bounded elastic solid embedded in an unbounded fluid medium. The elastic and acoustic domains are coupled through transmission conditions derived from the continuity of the normal stress and of the normal component of the velocities at the interface. The analysis of the HDG discretization of the coupled weak formulation is the main focus of the article. Two mixed variables (the stress tensor and the velocity of the acoustic wave) are included, while the symmetry of the stress tensor is imposed weakly by considering the antisymmetric part of the strain tensor (the spin or vorticity tensor) as an additional unknown. Convergence of the method is demonstrated and theoretical rates are obtained; numerical results suggesting optimal order of convergence and superconvergence of the traces are presented.


\end{abstract}

\section{Introduction}

We are interested in the computational simulation of the interaction between a transient acoustic wave and a homogeneous, isotropic and linearly elastic solid. The physical setting of the problem is as follows. An incident acoustic wave, represented by its scalar velocity potential $\vinc$, propagates at constant speed $c$ in a homogeneous, isotropic and irrotational fluid with density $\rA$ filling a region $\oma$ and impinges upon an elastic body of density $\rE$ contained in a bounded region $\ome$ with Lipschitz boundary $\ga$ and exterior unit normal vector $\nore$. Part of the energy and momentum carried by the acoustic wave is transferred to the elastic solid, exciting an internal elastic wave $\uv$, while the remaining momentum and energy are carried by an acoustic wave $v$ that is scattered off the surface $\ga$ of the elastic body. The physical setting is represented graphically in the left panel of \autoref{fig:unbounded}.

Due to the linearity of the problem, the total acoustic wave $\vtot = \vinc + v$ is the superposition of the known incident field $\vinc$ and the unknown scattered field $v$. The unknowns are thus the scattered acoustic field $v$ and the excited elastic displacement field $\uv$ that satisfy the following system of time-dependent partial differential equations \cite{tesis_tonatiuh}:
\begin{align*}
    -\midiv\lef(2\mu\strain{\uv}+\lambda\midiv\uv\bos{I}\rig)+\rE\ddot{{\uv}} &= \efe &\quad&\text{in }\ome,\\
    -\Delta v+c^{-2}\ddot{v}&=f &\quad&\text{in }\oma,\\
    -\nabla  \vtot\cdot\nore+\dot{\uv}\cdot\nore&=0 &\quad&\text{on }\ga,\\
    -\rA\dot{v}^{\text{tot}}\nore+\lef(2\mu\strain{\uv}+\lambda\midiv\uv\bos{I}\rig)\nore &= \bos{0} &\quad&\text{on }\ga,
\end{align*}
including suitable initial and radiation conditions, where the upper dot represents differentiation with respect to time, $\strain{\uv}:=\frac{1}{2}(\nabla\uv+\nablat\uv)$ is the strain tensor, $\bos{I}$ is the identity tensor, $\efe$ and $f$ are square integrable source terms for every time. The Lam\'e constants, $\mu$ (shear modulus) and $\lambda$ (Lam\'e's first parameter), encode the material properties of the solid and are assumed to be such that $\lambda>0$ and $2\mu+\lambda>0$. The symmetric tensor
\begin{equation}\label{eq:Stress}
\boldsymbol\sigma := 2\mu\strain{\uv}+\lambda\midiv\uv\bos{I}
\end{equation}
is known as the Cauchy stress tensor and can be represented compactly as $\boldsymbol\sigma = \mathbf C\boldsymbol\varepsilon(\boldsymbol u)$, where Hooke's elasticity tensor $\hooke$ is defined by its action on an arbitrary square matrix $\ten{M}$ as
\[
\hooke\ten{M}:= 2\mu \ten{M} + \lambda {\text{tr}}(\ten{M})\ten{I} 
\quad \text{ and } \quad
\hookei(\ten{M}) :=\frac{1}{2\mu}\ten{M} - \frac{\lambda}{2\mu(n\lambda + 2\mu)}{\text{tr}}(\ten{M})\ten{I},
\]
where $\text{tr}(\ten{M}):=\sum_{i=1}^{n} M_{ii}$ is the matrix trace operator. We will follow the approach from \cite{Ke}, where the symmetry of the stress tensor $\boldsymbol\sigma$ is imposed weakly by introducing the spin (or vorticity) tensor
\[
\g(\uv) := (\nabla \uv - \nablat \uv)/2
\]
as an additional unknown. The ``constant coefficients" assumption on the material densities and Lam\'e parameters has been made for simplicity and to ease some of the computations ahead. However, the results still hold true (with no more significant work) if for all times these parameters are considered to be essentially bounded functions such that the inequalities
\[
\rho_f>0\,,\quad \rho_E>0\,, \quad \lambda >0\,,\quad  \text{ and } \quad 2\mu+\lambda >0
\]
hold almost everywhere in the respective domains. The assumption of isotropy implicit in the definition of the stress tensor \eqref{eq:Stress} can also be relaxed significantly. Indeed it is enough to demand from its components $\mathbf C_{ijkl}$ to be essentially bounded functions for all time that satisfy the following symmetry and positivity requirements
\[
\mathbf C_{ijkl} = \mathbf C_{klij} = \mathbf C_{jilk} \qquad \text{ and } \qquad \mathbf M :\mathbf C \mathbf M \geq c_0  \mathbf M : \mathbf M 
\]
for some $c_0>0$ and every matrix $\mathbf M$.

\begin{figure}
    \centering
    \includegraphics[width=0.33\linewidth]{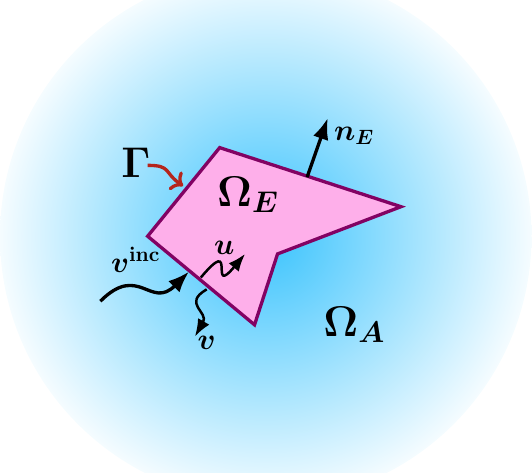} \qquad \qquad \qquad
    \includegraphics[width=0.33\linewidth]{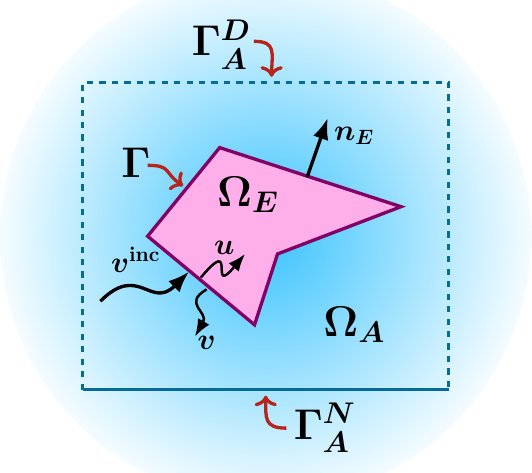} 
    \caption{Schematic representation of the problem geometry. Left: The elastic domain $\Omega_E$ is bounded; its Lipschitz boundary is denoted as $\Gamma$. The domain $\Omega_A$ where the acoustic waves propagate is unbounded. Right: An artificial boundary $\Gamma_A$ enclosing the elastic domain is introduced. The artificial boundary is split onto two disjoint components  $\Gamma_A^D$ and $\Gamma_A^N$ where Dirichlet and Neumann boundary conditions are imposed respectively. }
    \label{fig:unbounded}
\end{figure}

When viewed in full generality, the acoustic propagation region $\oma$ is in fact unbounded and given by $\oma:=\mathbb R^n\setminus\overline{\ome}$. This fact introduces further computational challenges that are often addressed either through an integral equation representation of the acoustic wave \cite{BrSaSa:2016,HsSa:2021,HsSaSa:2016,SV:2024}, the introduction of a perfectly matched layer \cite{LiTa:1997}, the use of absorbing boundary conditions \cite{EnMa:1977,HaTr:1988,Hi:1987,HaTr:1986} or the representation of the acoustic field through a moment expansion \cite{AlAn:2024}. 

In this communication, we simplify the analysis by introducing an artificial boundary that will allow us to assume that the acoustic domain is in fact bounded. As depicted in the right panel of \autoref{fig:unbounded}, we pick a polygon with boundary $\Gamma_A$  (the subscript standing for ``artificial") that compactly contains the elastic domain $\Omega_E$. The boundary $\Gamma_A$ is divided into mutually disjoint Dirichlet and Neumann segments (denoted respectively by $\gaadir$ and $\gaaneu$) such that $\gaa = \gaadir \cup \gaaneu$. The acoustic domain $\Omega_A$ is then defined to be the region exterior to $\Omega_E$ and contained inside the polygon. Its boundary takes the form
\[
\partial \Omega_A : = \Gamma \cup \Gamma_A^D \cup \Gamma_A^N
\]
where the three components are mutually disjoint and $\Gamma$ denotes the interface between the acoustic and elastic regions. We emphasize that the boundary conditions imposed on $\gaa$ do not attempt to account for a physically outgoing wave, but simply to ensure the well-posedness of the simplified problem. The goal of this work is to establish the well-posedness theory for the coupling of HDG discretizations for elastic and acoustic wave propagation. The treatment of the fully unbounded problem with appropriate outgoing boundary conditions will be the subject of a separate communication, nevertheless, even in this simplified form, the analysis of the coupled HDG discretization has not been tackled in the literature as far as the authors are aware of. 

Assuming that at the initial time the incident wave $v^{\text{inc}}$ is supported away from the elastic domain $\Omega_E$, the distributional version of the system above admits a Laplace transform \cite{HsSaSa:2016} that maps time differentiation to multiplication by the Laplace parameter $s\in \{z\in \C: \re{z}>0\}$. Upon Laplace transformation and using the same symbols for the unknowns in the time domain and in the Laplace domain, the elastic wave $\uv$ and the scattered acoustic wave $v$ satisfy the coupled system of equations in mixed form 
\begin{subequations}\label{Equation1}
\begin{alignat}{6}
\label{Equation1A}
\hookei\si-\nabla \uv+ \g& =\ten{0} &\qquad & \text{ in }\ome ,\\
\label{Equation1B}
-\midiv\si +\rE s^2\uv& =\efe & \qquad &  \text{ in }  \ome,\\
\label{Equation1C}
\q - \nabla v & = \bos{0} & \qquad & \text{ in } \oma, \\
\label{Equation1D}
-\midiv\q+(s/c)^2v& =f\hspace{0.5cm} & \qquad & \text{ in } \oma,\\
\label{Equation1E}
\q \cdot \nora + s\, \uv\cdot \nore & = -\nabla \vinc \cdot \nora & \qquad & \text{ on }  \ga,\\
\label{Equation1F}
- \si  \nore - \rA s\, v \,\nora & = \rA s\, \vinc \,\nora & \qquad & \text{ on }  \ga, \\
\label{Equation1G}
v & = g_D & \qquad & \text{ on }  \gaadir,\\
\label{Equation1H}
\q\cdot \nora & = g_N & \qquad & \text{ on }  \gaaneu.
\end{alignat}
\end{subequations}

Here, $\boldsymbol q$ is the acoustic velocity field, and $g_D \in H^{1/2}(\gaadir)$ and $g_N \in H^{-1/2}(\gaaneu)$ are given boundary data. 

In the system above, equations \eqref{Equation1A} and \eqref{Equation1B} account for the Navier-Lam\'e or elastic wave equation in the interior of the elastic solid $\ome$. Similarly, equations \eqref{Equation1C} and \eqref{Equation1D} are the mixed form of the acoustic wave equation in $\oma$. The elastic and acoustic variables are coupled through the continuity of the normal component of the velocity field across the interface $\ga$, encoded in equation \eqref{Equation1E}, and the balance of normal forces at the contact surface, given in \eqref{Equation1F}. The nonphysical boundary conditions \eqref{Equation1G} and \eqref{Equation1H} prescribed at the artificial boundary $\gaa$ are given to ensure the well-posedness of the problem.

In the literature, there is a vast amount of research related to fluid-structure interaction problems. For instance, some of them use a Mixed Finite Elements approach \cite{DoGaMe:2015,GaMaMe:2012} and there are also couplings of this technique with Boundary Element Methods \cite{GaHeMe:2014}. Studies on their spectral problems \cite{MeMoRo:2014} and an analysis of the elastoacustic problem in the time domain \cite{ArRoVe:2019} have been done. However, most of these works assume a time-harmonic regime, while we intend to focus on the transient regime. A notable time--domain contribution is the very recent contribution \cite{Mottier2025}, where a Hybrid High Order (HHO) method is used for a similar acoustic/elastic interaction in the time domain.

Since two different systems of PDEs posed in different domains are being coupled across an interface, we prefer to use a discontinuous Galerkin scheme due to its flexibility to handle the transmission conditions. In particular, by considering the HDG method introduced in \cite{CoGoLa2009}, it is very easy to impose transmission conditions from the computational point of view. In fact, in HDG schemes the only globally coupled degrees of freedom are precisely those of the numerical traces on the boundaries between elements, while the remaining unknowns are obtained by solving local problems in each element. Therefore, if we have two independent HDG solvers, one for the acoustic problem and another one for the elastic system, we can couple them across the interface through the numerical traces associated with the acoustic wave $v$ and the elastic displacements $\boldsymbol{u}$.   

After \cite{CoGoLa2009} and the pioneering work \cite{CoGoSa2010} that set a framework that simplifies the analysis of a family of HDG schemes by introducing a suitable projection, HDG schemes have been developed for a wide variety of problems. For example, convection-diffusion equation \cite{FuQiZh2015,NgPeCo2009}, Stokes flow \cite{CoGoNgPe2011,GaSe2016}; Brinkman, Oseen and Navier--Stokes equations \cite{CeCoNgPe2013,CeCoQi2017,FuJiQi2016,NgPeCo2011NS}. In the context of electromagnetism and wave propagation problems, HDG schemes have also been introduced: Maxwell's operator \cite{ChQiSh2018,ChQiShSo2017}, eddy current problems \cite{BuLoOs2018}, Maxwell's equations in the frequency-domain \cite{FePeXu2016,NgPeCo2011} and heterogeneous media \cite{CaLoOsSo2020} and Helmholtz equation \cite{ChPeXu2013,GrMo2011,ZhWu2021}, and even for nonlinear problems arising from plasma physics \cite{SaSaSo:2021a,SaSaSo:2019,SaSo:2018,SaSoCe:2018}. For the elasticity problem, we refer the reader to \cite{Ke,QiShSh2018}. The preceding list of references is not exhaustive, but provides an overview of the development of HDG schemes during the last fifteen years.

On the other hand, in the context of coupled problems with piecewise linear interfaces, HDG schemes have been proposed for elliptic \cite{Huynh2013} and for the Stokes interface problems \cite{Wang2013HDG}, and for Stokes-Darcy coupling \cite{GaSe2017}. The influence of hanging-nodes along the interface and the use of different polynomial degree over each local space, have been analyzed in \cite{Chen2012,Chen2013}. Recently, a new approach based on the Transfer Path Method \cite{CoQiSo2014,CoSo2012,SaSaSo2022} has been proposed to handle discrete interfaces that do not necessarily coincide with the true interface, as in the case of a curved interface \cite{BeMaSo2025,MaNgSo2022,SoTeNgPe2022}. This technique produces a high order method and is closely related with our ultimate goal, where it is crucial to have a numerical scheme that couples an HDG discretization of the problem posed in an bounded domain considering a solid with a curved boundary, and a representation of the acoustic wave in the unbounded region. To the best of our knowledge, the use of HDG schemes has not been analyzed for the coupled problem \eqref{Equation1}, and the main contribution of this work is to provide a convergence analysis.

%
\section{Preliminaries and notation}
%
\subsection{Sobolev spaces.}

Let $\mathcal{O}$ be a Lipschitz continuous domain in $\mathbb{R}^{n}$. We use standard notations for Lebesgue $\mathrm L^p(\mathcal{O})$ and Sobolev spaces $\mathrm W^{q,p}(\mathcal{O})$, with $q\, \geq \, 0$ and $p\in [1,+\infty)$. Here $\mathrm W^{0,p}(\mathcal{O}) \,=\, \mathrm L^p(\mathcal{O})$, and if $p\,=\,2$ we write $\mathrm H^q(\mathcal{O})$ instead of $\mathrm W^{q,2}(\mathcal{O})$, with the corresponding norm and seminorm  denoted by $\|\cdot\|_{H^{q}(\mathcal{O})}$ and $|\cdot|_{H^{q}(\mathcal{O})}$, respectively. The spaces of vector-valued functions will be denoted in boldface, therefore $\boldsymbol{H}^{s}(\mathcal{O}) := [H^{s}(\mathcal{O})]^{n}$, whereas for tensor-valued functions, we write $\underline{\boldsymbol{H}}^{s}(\mathcal{O}) := [H^{s}(\mathcal{O})]^{n\times n}$. Using the same notation, we write $\bos{L}^2(\mathcal{O}):=[L^2(\mathcal{O})]^n$ and $\uline{\bos{L}}^2(\mathcal{O}):=[L^2(\mathcal{O})]^{n\times n}.$

The complex $L^2$-inner products will be denoted by $(\cdot,\cdot)_{\mathcal O}$ and $\langle\cdot,\cdot\rangle_{\Sigma}$, where $\Sigma$ is either a Lipschitz curve ($n=2$) or a surface ($n=3$). The associated norms will be denoted by $\|\cdot\|_{\mathcal O}$ and $\|\cdot\|_{\Sigma}$. 

It is easy to verify that Hooke's tensor satisfies the following inequalities for all $\etat\in\uline{\bos{L}}^2(\mathcal{O})$:
\begin{align*}
\lef(\dfrac{1}{2\mu}+\dfrac{ n^2\lambda}{2\mu(n\lambda+2\mu)}\rig)^{-1}\|\etat\|_{\mathcal{O},\hookei}^2 &\leq\|\etat\|^2_{\mathcal{O}}\leq 2\mu\|\etat\|_{\mathcal{O},\hookei}^2,\\[1ex]
\|\etat\|_{\mathcal{O},\hooke}^2 &\leq(2\mu+n^2\lambda)\|\etat\|^2_{\mathcal{O}},\end{align*}
where we denote $\|\cdot\|_{\mathcal{O},\hookei}:=(\hookei\cdot,\cdot)_{\mathcal{O}}^{1/2}$ and $\|\cdot\|_{\mathcal{O},\hooke}:=(\hooke\cdot,\cdot)_{\mathcal{O}}^{1/2}$.
%
\subsection{Mesh and mesh-dependent inner products.}
%
Let $\ta$ and $\te$ be two families of regular triangulations of $\oma$ and $\ome$, respectively. We will assume that these triangulations are compatible along the common interface $\ga$ and that both are characterized by a common mesh size $h$ in their respective domains. Given an element $K$, $h_K$ will denote its diameter and $\nor_K$ its outward unit normal. When there is no confusion, we will simply write $\nor$ instead of $\nor_K$. Set $\dagger\in\{A,E\}$, then $\partial\mcal{T}_\dagger:=\{\partial K:K\in\mcal{T}_\dagger\}$ and let $\mcal{E}_\dagger$ denote the set of all faces $F$ of all elements $K\in\mcal{T}_\dagger$. We will also use the following notation for $L^2$ inner products of scalar-, vector- and tensor-valued functions, respectively, over an integration domain $D$:
\[
(u,v)_{D} : = \int_D u\overline{v},\qquad (\bos u,\bos v)_D = \int_D \bos u\cdot\overline{\bos v}, \qquad (\bos M,\bos N)_D = \int_{D} \bos M:\overline{\bos N},
\]
where the overline denotes complex conjugation and the colon ``$\,:\,$" is used to denote the Frobenius inner product of matrices
\[
\bos M:\bos N : = \sum_{i,j=1}^n M_{ij}N_{ij}.
\]
With this notation we can express the mesh-dependent $L^2$ inner products as
\[
\lef(u,v\rig)_{\mcal{T}_\dagger}:=\dis\sum_{K\in\mcal{T}_\dagger}   \lef(u,v\rig)_{K}, \lef(\bos u,\bos v\rig)_{\mcal{T}_\dagger}:=\dis\sum_{K\in\mcal{T}_\dagger} \lef(\bos u,\bos v\rig)_{K}, \quad \lef(\ten{M},\ten{N}\rig)_{\mcal{T}_\dagger}:=\dis\sum_{K\in\mcal{T}_\dagger} \lef(\ten{M},\ten{N}\rig)_{K},
\]
along with the inner products over the mesh skeleton 
\[
\langle u,v\rangle_{\partial\mcal{T}_\dagger}:=\!\dis\sum_{K\in\mcal{T}_\dagger}\!\lef\langle u,v\rig\rangle_{\partial K}, \langle\bos u,\bos v\rangle_{\partial\mcal{T}_\dagger}:=\!\dis\sum_{K\in\mcal{T}_\dagger} \!\lef\langle\bos u,\bos v\rig\rangle_{\partial K}, \langle\ten{M},\ten{N}\rangle_{\partial\mcal{T}_\dagger}:=\!\dis\sum_{K\in\mcal{T}_\dagger} \!\lef\langle\ten{M},\ten{N}\rig\rangle_{\partial K}.
\]
We denote the norms induced by these inner products by
\[
\|\cdot\|_{\mcal{T}_\dagger}:=\sqrt{\lef(\cdot,\cdot\rig)_{\mcal{T}_\dagger}} \qquad \text{ and } \qquad \|\cdot\|_{\partial\mcal{T}_\dagger}:=\sqrt{\langle\cdot,\cdot\rangle_{\partial\mcal{T}_\dagger}}.
\]
Finally, to avoid proliferation of superfluous constants, we will write $a \lesssim b$ when there exists a positive constant $C$, independent of the mesh size, such that $a\leq C b$.

\subsection{The HDG polynomial spaces}

We will make use of the discrete spaces for the HDG method proposed in \cite{Ke} for simplices. For an element $K\in \ta\cup\te$, we define the following function spaces. The set of scalar-valued polynomials of degree at most $k$ defined over $K$ will be denoted by $\pk{k}{K}$, while the corresponding vector and tensor product spaces are denoted respectively as
\[
\pkv{k}{K}:=[\pk{k}{K}]^{n} \qquad \text{ and } \qquad \pkt{k}{K}:=[\pk{k}{K}]^{n\times n}.
\]
The polynomial spaces of degree \textit{exactly} $k$ will be denoted with a tilde as $\widetilde{\mathcal P}_k(K)$, $\widetilde{\boldsymbol{\mathcal P}}_k(K)$, and $\widetilde{\uline{\boldsymbol{\mathcal P}}}_k(K)$. We will denote the space of $L^2$ integrable skew-symmetric matrices over $K$ by
\[
 \asspace{K}:=\lbrace {\bos M} \in \ldt{K}:{\bos M}+{\bos M}^{\top}={\bos{0}}\rbrace,
\]
Our discretization will make use of spaces of skew-symmetric matrices with polynomial entries. The skew-symmetry can be enforced constructively by defining the skew-symmetric-valued polynomial space
\[
\aspace{K}:= \left\{A_{ij}(K): i,j=1,\ldots,n.\right\} \quad \text{ where } \quad A_{ij}\lef(K\rig):=
	\begin{cases} 
       \pk{k}{K} & \text{if} \hspace{0.3cm} i > j,  \\
       0 &  \text{if}  \hspace{0.3cm}i = j, \\
       -\pk{k}{K} & \text{if} \hspace{0.3cm} i > j,
    \end{cases},
\]
which by construction satisfies $\aspace{K} \subset \asspace{K}$.

Now, we would like to define a divergence-free space of functions through the use of bubble matrices or bubble scalars, depending on the dimension, as in \cite{CaSo:2023,CoGoGu:2009,Ke,Gu:2010}. Following \cite{Gu:2010}, a matrix-valued function $\bos{b}$ defined in $\ome$ is said to be an admissible bubble matrix if for each $K\in\te$ the matrix $\bos{b}_K:=\ra{\bos{b}}{K}$ is a matrix with polynomial entries that satisfies
\begin{enumerate}
    \item The tangential components of each row of $\bos{b}_K$ vanish on all the faces of $K$,
    \item There exists $C_1>0$ such that $C_1\ipk{\bos{v}}{\bos{v}}\leq\ipk{\bos{vb}_K}{\bos{v}}$, for all $\bos{v}\in\uline{\bos{L}}^2(K)$,
    \item There exists $C_2>0$ such that $\|\bos{b}_K\|_{\uline{\bos{L}}^\infty(K)}\leq C_2$,
\end{enumerate}
where the constants $C_1$ and $C_2$ depend only on the shape regularity of $\te$. 

Thus, following \cite{CoGoGu:2009,Ke}, if $\eta_F$ is the barycentric coordinate associated to the edge $F$ of $K$, and if we define
\[
\bos{b}_K:= { \begin{cases} \displaystyle\prod_{F \subset \partial K} \eta_{F} & \text{in 2D }, \\[1ex]  \displaystyle\sum_{F \subset \partial K} \lef[ \displaystyle\prod_{F' \subset \partial K \setminus \lbrace F \rbrace}\eta_{F'}\rig]\nabla\eta_{F}\otimes\nabla\eta_{F} & \text{in 3D},\end{cases}}
\]
the polynomial space $\bspace{K}$ associated to bubble functions is defined as:
\[
\bspace{K}:= \nabla\times((\nabla\times\aspace{K})\bos{b}_K).
\]
We can observe that any function
\[
\bos{v}\in\uline{\bos{\mathcal{B}}}_h:=\{\etat\in\uline{\bos{L}}^2(\ome): \ra{\etat}{K}\in\bspace{K}, K\in\te\}
\]
is such that
\[
\ra{\midiv\bos{v}}{K}=0,\forall\, K\in\te \quad \text{ and } \quad \ra{\bos{v}\nor}{F}=0,\forall F\in\ee.
\]
In the three-dimensional case the curl operator acts row-wise, while in the two-dimensional case the curl of matrices and column vectors are defined respectively by
\[
\nabla\times\begin{pmatrix}
    M_{11} & M_{12}\\ M_{21} & M_{22}
\end{pmatrix}:=\begin{pmatrix}
    \partial_xM_{12}-\partial_yM_{11}\\\partial_xM_{22}-\partial_yM_{21}
\end{pmatrix}\quad\text{and}\quad \nabla\times \begin{pmatrix}
    m_1\\ m_2
\end{pmatrix}:=\begin{pmatrix}
    -\partial_y m_1 & \partial_x m_1\\
    -\partial_y m_2 & \partial_x m_2
\end{pmatrix}
\]

We will also make use of the local space $
\vtspace{K} :=\pkt{k}{K}+\bspace{K}$, and notice that
\[
\vtspace{K} = \pkt{k}{K} +  \nabla \times ((\nabla \times \aspace{K}) \bos{b}_K) = \pkt{k}{K}\oplus \nabla \times ((\nabla \times \atilspace{K} )\bos{b}_K),
\]
where $\atilspace{K} := \aspace{K} \cap \pkttil{k}{K}$.

%
\section{An HDG discretization}
%
Let us begin by introducing the piecewise polynomial spaces
\begin{subequations}\label{spaces}
\begin{align}
\vhspace&=\lbrace\taut \in \ldt{\te}: \ra{\taut}{K}\in \vtspace{K},\hspace{0.3cm}\forall K \in \te \rbrace, \\
\whvespace&=\lbrace\tv \in \ldv{\te}: \ra{\tv}{K}\in\pkv{k}{K},\hspace{0.3cm} \forall K \in \te \rbrace, \\
\ahspace&=\lbrace\etat\in \ldt{\te}: \ra{\etat}{K}\in \aspace{K},\hspace{0.3cm} \forall K \in \te \rbrace, \\
\mhvspace&=\lbrace\muv \in \ldv{\ee}: \ra{\muv}{F}\in \pkv{k}{F},\hspace{0.3cm} \forall F \in \ee \rbrace,\\
\whvaspace&=\lbrace\rv \in \ldv{\ta}: \ra{\rv}{K}\in\pkv{k}{K},\hspace{0.3cm} \forall K \in \ta \rbrace, \\
\whspace&=\lbrace w\in \ld{\ta}: \ra{w}{K}\in \pk{k}{K},\hspace{0.3cm} \forall K \in \ta \rbrace, \\
\mhspace&=\lbrace\xi\in \ld{\ea}: \ra{\xi}{F}\in \pk{k}{F},\hspace{0.3cm} \forall F \in \ea \rbrace.
\end{align}
\end{subequations}
The HDG discretization seeks a piecewise polynomial approximation 
\[
(\sih,\uvh,\gh,\uvhg,\qh,\vh,\vhg)\in \vhspace\times\whvespace\times\ahspace\times\mhvspace\times\whvaspace\times \whspace\times \mhspace
\]
of the exact solution $(\si,\uv,\g,\ra{\uv}{\ee},\q,v,\ra{v}{\ea})$. The approximation must satisfy the discrete weak formulation
\begin{subequations}\label{hdgscheme}
\begin{align}
\ipte{\hookei\sih}{\taut}+\ipte{\uvh}{\midiv \taut}+\ipte{\gh}{\taut}-\ipbe{\uvhg}{\taut\nor}&=0,\label{seqh}\\[1ex]
\ipte{\sih}{\nabla \tv}-\ipbe{\sihg \nor}{\tv} +\rE s^2\ipte{\uvh}{\tv}& = \ipte{\efe}{\tv},\label{seqi}\\[1ex]
\ipte{\sih}{\etat} & = 0,\label{seqj}\\[1ex]
\ipfe{\sihg\nor}{\muv} & = 0,\label{seqk}\\[1ex]
\ipta{\qh}{\rv}+\ipta{\vh}{\midiv \rv}-\ipba{\vhg}{\rv \cdot\nor}&=0,\label{HDG:acus_1}\\[1ex]
\ipta{\qh}{\nabla w}-\ipba{\qhg\cdot \nor}{w} +(s/c)^2\ipta{\vh}{w}& = \ipta{f}{w},\label{HDG:acus_2}\\[1ex]
\langle \qhg\cdot \nor,\xi\rangle_{\dta \setminus(\ga \cup \gaadir)} & = \langle g_N ,\xi\rangle_{\gaaneu} ,\label{seqn}\\[1ex]
\langle \vhg,\xi\rangle_{\gaadir} & = \langle g_D ,\xi\rangle_{\gaadir} ,\label{seqo}\\[1ex]
\label{Equation1E_h}
\ipg{\qhg \cdot \nora + s\, \uvhg\cdot \nore}{\xi} & = -\ipg{\nabla \vinc \cdot \nora}{\xi},\\[1ex]
\label{Equation1F_h}
\ipg{-\sihg \nore - \rA s\, \vhg \,\nora}{\muv} & = \rA s\, \ipg{\vinc\,\nora}{\muv}
\end{align}
for all test functions $(\taut,\tv,\etat,\muv,\rv,w,\xi) \in \vhspace\times\whvespace\times\ahspace\times\mhvspace\times\whvaspace\times \whspace\times \mhspace$, where
\begin{align}
\label{eq:15k}
\sihg \nor&:= \sih\nor-\tau_E(\uvh-\uvhg)\hspace{1cm} \text{on} \hspace{0.5cm} \dte , \\
\qhg\cdot \nor&:= \qh \cdot \nor-\tau_A(\vh-\vhg)\hspace{1cm} \text{on} \hspace{0.5cm} \dta. 
\end{align}
Here, $\tau_E$ and $\tau_A$ are stabilization parameters whose properties will be determined when analyzing the scheme.
\end{subequations}

\section{Discrete well posedness.}
\begin{theorem}
\label{ExandUniq}
    If $\re{s\tau_A}>0$ and $\re{s\tau_E}>0$, then the scheme \eqref{hdgscheme} has a unique solution. 
\end{theorem}
\begin{proof}
By the Fredholm alternative, it is enough to show uniqueness of the solution. To that end, if we assume zero sources, we will show that the solution to the corresponding system is the trivial one.

Let
\[
\vinc = 0 \qquad \text{ and } \qquad (f,\efe, g_D, g_N) = (0, \bos{0}, 0, 0),
\]
and choose
\[
(\taut , \tv ,\etat , \muv , \rv , w) = (\sihg, \uvh, \gh, \uvhg, \qh, \vh)
\qquad \text{ and } \qquad
\xi = \left\{ \begin{array}{rl} \vhg, & \text{on }\dta\setminus \gaadir \\ \qhg\cdot\nor, & \text{on }\gaadir \\\end{array} \right.\,.
\]
With this choice of test functions, applying integration by parts to \eqref{seqi} and adding its conjugate to \eqref{seqh} we obtain
\begin{align*}
       &\ipte{\hookei\sih}{\sih} + \ipte{\uvh}{\midiv\sih} + \ipte{\gh}{\sih} - \ipbe{\uvhg}{\sih\nor}\\
       &-\conj{\ipte{\midiv\sih}{\uvh}} + \conj{\ipbe{\sih\nor}{\uvh}} - \conj{\ipbe{\sihg\nor}{\uvh}} + \rE\conj{s^2}\ipte{\uvh}{\uvh} = 0.
\end{align*}
We know from \eqref{seqj} that $\ipte{\sih}{\gh} = 0$, so the latter equation becomes
\begin{align*}
       \ipte{\hookei\sih}{\sih}+\conj{\ipbe{\sih\nor-\sihg\nor}{\uvh}}
       -\ipbe{\uvhg}{\sih\nor}+\rE\conj{s^2}\ipte{\uvh}{\uvh} = 0.
\end{align*}  
Adding and subtracting $\uvhg$ in the second argument of the second term, we have that
\begin{align*}
\ntehookei{\sih}^2 &+\conj{\ipbe{\sih\nor-\sihg\nor}{\uvh-\uvhg}}+\conj{\ipbe{\sih\nor-\sihg\nor}{\uvhg}}\\
       &-\ipbe{\uvhg}{\sih\nor}+\rE\conj{s^2}\nte{\uvh}^2 = 0.
\end{align*}
Multiplying by $s$ and using \eqref{seqk}, along with the definition \eqref{eq:15k}, we obtain
\begin{equation}
    s\ntehookei{\sih}^2 + s\conj{\ipbe{\tau_E(\uvh-\uvhg)}{\uvh-\uvhg}}-s\ipg{\uvhg}{\sihg\nore}+\rE \overline{s}|s|^2\nte{\uvh}^2=0. \label{eau:elastic}
\end{equation}
Analogously for the acoustic terms, \eqref{HDG:acus_2} is integrated by parts and its conjugate is added to \eqref{HDG:acus_1}, yielding
\begin{align*}
    \nta{\qh}^2+&\ipta{\vh}{\midiv\qh}-\ipba{\vh}{\qh\cdot\nor}-\conj{\ipta{\midiv\qh}{\vh}}\\
&+\conj{\ipba{\qh\cdot\nor}{\vh}}-\conj{\ipba{\qhg\cdot\nor}{\vh}}+\frac{\conj{s^2}}{c^2}\nta{\vh}^2=0.
\end{align*}
Adding and subtracting $\vhg$ and using \eqref{seqn} and \eqref{seqo}, we can deduce that 
\[
    \nta{\qh}^2+\conj{\ipba{\tau_A(\vh-\vhg)}{\vh-\vhg}}-\ipg{\vhg}{\qhg\cdot\nora}+\frac{\conj{s^2}}{c^2}\nta{\vh}^2=0.
\]
We multiply the latter equation by $\rA s$ to obtain
\begin{equation}
\rA s\nta{\qh}^2+\rA s\conj{\ipba{\tau_A(\vh-\vhg)}{\vh-\vhg}}-\rA s\ipg{\vhg}{\qhg\cdot\nora}+\rA \overline{s}(|s|/c)^2\nta{\vh}^2=0. \label{eau:acoustic}
\end{equation}
Adding \eqref{eau:elastic} with the conjugate of \eqref{eau:acoustic} leads to
\begin{equation}\label{eq:temp1}
\begin{array}{l}
     \;\;\,s\ntehookei{\sih}^2 + s\conj{\ipbe{\tau_E(\uvh-\uvhg)}{\uvh-\uvhg}}-s\ipg{\uvhg}{\sihg\nore}+\rE \overline{s}|s|^2\nte{\uvh}^2\\
     
    +\rA \conj{s}\nta{\qh}^2+\rA \conj{s}\ipba{\tau_A(\vh-\vhg)}{\vh-\vhg}-\rA \conj{s}\conj{\ipg{\vhg}{\qhg\cdot\nora}}+\rA \overline{s}(|s|/c)^2\nta{\vh}^2=0
    \end{array}
\end{equation}
Notice that from \eqref{Equation1E_h} and \eqref{Equation1F_h} we have
\begin{align*}
    -s\ipg{\uvhg}{\sihg\nore}-\rA\conj{s}\conj{\ipg{\vhg}{\qhg\cdot\nora}}
       = & -s\ipg{\uvhg}{\sihg\nore+\rA s\vhg\nora}+s\ipg{\uvhg}{\rA s\vhg\nora}\\&-\rA\conj{s}\conj{\ipg{\vhg}{\qhg\cdot\nora+s\uvhg\cdot\nore}}+\rA\conj{s}\conj{\ipg{\vhg}{s\uvhg\cdot\nore}}\\
    = & \phantom{+}s\ipg{\uvhg}{\rA s\vhg\nora}+\rA\conj{s}\conj{\ipg{\vhg}{s\uvhg\cdot\nore}}\\
       = & \phantom{+}s\conj{s}\rA\ipg{\uvhg}{\vhg\nora}-s\conj{s}\rA\ipg{\uvhg}{\vhg\nora}
       =  \; 0.
\end{align*}
So, \eqref{eq:temp1} is equivalent to
\begin{align*}
    &s\ntehookei{\sih}^2 + s\conj{\ipbe{\tau_E(\uvh-\uvhg)}{\uvh-\uvhg}}+\rE \overline{s}|s|^2\nte{\uvh}^2\\
+&\rA \conj{s}\nta{\qh}^2+\rA \conj{s}\ipba{\tau_A(\vh-\vhg)}{\vh-\vhg}+\rA s(|s|/c)^2\nta{\vh}^2=0.
\end{align*}
Thus, taking real part of this expression, we obtain
\[
\mathscr{E}_E^2 + \mathscr{E}_A^2 + \rE|s|^2\re{s}\nte{\uvh}^2+\frac{\rA}{c^2}|s|^2\re{s}\nta{\vh}^2 = 0\,,
\]
where we have defined
\begin{align*}
    \mathscr{E}_E &:= \sqrt{\ntehookei{\re{s}^{1/2}\sih}^2 + \nbte{\re{s\tau_E}^{1/2}(\uvh-\uvhg)}^2}\\[1ex]
    \mathscr{E}_A &:= \sqrt{ \nta{\rA^{1/2}\re{s}^{1/2}\qh}^2+\nbta{\rA^{1/2}\re{s\tau_A}^{1/2}(\vh-\vhg)}^2}.
\end{align*}
From here, we can conclude that $\sih = \ten{0}$ in $\te$, $\uvh = \bos{0}$ in $\te$, $\qh = \bos{0}$ in $\ta$, $\vh = 0$ in $\ta$, $\uvhg = \uvh = \bos{0}$ on $\dte$ and $\vhg = \vh = 0$ on $\dta$.

It only remains to show that $\gh = \bos{0}$ in $\te$. This will be achieved by performing an analog of the steps done in the proof of \cite[Lemma 3.6]{CaSo:2023}. We will need the two following technical results proven in \cite{Gu:2010}:
\begin{enumerate}
\item \cite[Lemma 2.8]{Gu:2010} Given $\etat \in \ahzspace:=\{\etat\in\ahspace:\ipk{\etat}{\ten{v}}=0, \forall \ten{v}\in\pkt{0}{K}, \forall K\in\te\}$, there exists $\ten{v}\in \uline{\bos{\mathcal{B}}}_h$ such that
\[
\ten{P}\ten{v} = \etat \quad \text{ and } \quad \nte{\ten{v}}\leq C^0\nte{\etat}.
\]
Here $\ten{P}:\ldt{\ome}\to\ahspace$ is the $L^2$-projection onto $\ahspace$ and $C^0$ is a positive constant independent of $h$, arising from a Poincaré-type inequality and inverse estimates.

\item \cite[Proposition 2.9]{Gu:2010}  Given $\etat\in\ahcspace:=\ahspace\cap\pkt{0}{\te}$, there exists $\ten{v}\in\uline{\bos{H}}(\text{div};\ome)\cap\pkt{1}{\te}$ such that
\begin{equation}\label{eq:temp3}
\midiv\ten{v}=0\,, \qquad 
\ten{P}^c\ten{v}=\etat\,, \qquad
\text{ and } \qquad
\nte{\ten{v}} \leq C^c\nte{\etat}\,,
\end{equation}
where $\ten{P}^c$ is the $L^2$-projection onto $\ahcspace$, and $C^c>0$ is a constant independent of $h$.
\end{enumerate}
Let us consider the orthogonal decomposition
\[
\gh = \gh^0+\gh^c \quad  \text{ where } \quad \ra{\gh^c}{K}:=\frac{1}{|K|}\int_K \gh, \forall K\in\te\,\,\,\, \text{ (component-wise) }\quad \text{ and } \quad \gh^0=\gh-\gh^c.
\]
It is clear that $
\gh^0\in\ahzspace$ and $\gh^c\in\ahcspace$.

By \cite[Lemma 3.9]{Gu:2010}, there exists
\[
\ten{v}^0\in\uline{\bos{\mathcal{B}}_h}:=\{\etat\in\ldt{\ome}: \ra{\etat}{K}\in\bspace{K}, K\in\te\}\subset\vhspace
\]
such that
\begin{equation}\label{eq:temp2}
\ipte{\gh^0}{\ten{\rho}^0} = \ipte{\ten{v}^0}{\ten{\rho}^0} \quad \text{ for all }\,\ten{\rho}^0\in\ahspace.
\end{equation}
Taking $\taut = \ten{v}^0$ in \eqref{seqh}, we obtain
\[
\ipte{\gh^0+\gh^c}{\ten{v}^0} = 0.
\]
Now, considering $\ten{\rho}^0=\gh^c$, and the fact that the decomposition of $\gh$ is orthogonal in $\uline{\bos{L}}^2$, the two expressions above imply
\[
\ipte{\gh^0}{\ten{v}^0} = \ipte{\gh^0}{\gh^c}=0.
\]
Hence, taking $\ten{\rho}^0=\gh^0$ in \eqref{eq:temp2}, the equality above shows that $\ipte{\gh^0}{\ten{v}^0} = \nte{\gh^0}^2 = 0$, and we can conclude that $\gh^0 = \bos{0}$.

Finally, by the second property in \eqref{eq:temp3}, there exists $\ten{v}^c\in\ten{H}(\text{div};\ome)\cap\pkt{1}{\te}$ such that
\[
\ipte{\ten{v}^c}{\ten{\rho}^c}=\ipte{\gh^c}{\ten{\rho}^c} \quad \text{ for all }\,\ten{\rho}^c\in\ahcspace.
\]
Taking $\ten{\rho}^c=\gh^c$ in the expression above we have 
\begin{align} \label{eq:aux_rho2}\ipte{\gh^c}{\ten{v}^c} = \nte{\gh^c}^2.  
\end{align}

Now, recalling that  $\sih = \ten{0}$ and $\uvh = \bos{0}$ in $\te$ and $\uvhg=\bos{0}$ on $\dte$, choosing $\taut=\ten{v}^c$ in \eqref{seqh}, we have that 
$\ipte{\gh}{\ten{v}^c}=0$. Then, since $\gh^0=\bos{0}$, from \eqref{eq:aux_rho2} we conclude $\gh^c=\bos{0}$ in $\te$, and therefore $\gh=0$.
\end{proof}

\section{Error Analysis.}
\subsection{The HDG Projections.}
We will need the HDG projections defined in \cite{CoGoSa2010}. For the acoustic terms, the projected function is denoted by $\bos{\Pi_A}(\q,v) := \lef(\piwa\q,\piw v \rig)$, where $\piwa\q$ and $\piw v$ are the components of the projection in $\whvaspace$ and $\whspace$, respectively. The values of the projection on any simplex $K\in\ta$ are fixed when the components are required to satisfy the equations
\begin{alignat*}{6}
    \lef(\piwa\q,\rv\rig)_K &= (\q,\rv)_K,\qquad&&\forall\rv\in\pkv{k-1}{K},\\
    \lef(\piw v, w\rig)_K &= (v,w)_K,\qquad&&\forall w\in\pk{k-1}{K},\\
    \lef\langle\piwa\q\cdot\nor-\tau_A\piw v,\xi\rig\rangle_F &= \lef\langle\q\cdot\nor-\tau_A \pem v,\xi\rig\rangle_F,\qquad&&\forall\xi\in\pk{k}{F},
\end{alignat*}
for all faces $F$ of the simplex $K\in\ta$, where $\pem$ is the $L^2$ projection onto $F$. It was shown in \cite{CoGoSa2010} that, if $(\q,v)\in\bos{H}^{k+1}(K)\times H^{k+1}(K)$ and $\ra{\tau_A}{\partial K}$ is nonnegative and $\displaystyle \max_{\partial K} \tau_A>0$, the components of the projection satisfy the estimates
\begin{subequations}\label{eq:ProjectionEstimates1}
\begin{align}
    \lef\|\piwa\q-\q\rig\|_K &\lesssim h_K^{k+1}\lef(|\q|_{\bos{H}^{k+1}(K)}+|v|_{H^{k+1}(K)}\rig),\\
    \lef\|\piw v-v\rig\|_K&\lesssim h_K^{k+1}\lef(|v|_{H^{k+1}(K)}+|\midiv\q|_{H^k(K)}\rig).
\end{align}
\end{subequations}
Therefore, for the sake of simplicity, from now on we assume that $\tau_E$ and $\tau_A$ are positive functions.

For the elastic terms, on each element $K\in\te$, a component-wise version of the above projection is defined by $\bos{\Pi_E}(\si,\uv):=\lef(\piv\si,\piwe\uv\rig)\in\pkt{k}{K}\times\pkv{k}{K}$ where
\begin{alignat*}{6}
    \lef(\piv\si,\taut\rig)_K &= (\si,\taut)_K,\qquad&&\forall\taut\in\pkt{k-1}{K},\\
    \lef(\piwe \uv, \tv\rig)_K &= (\uv,\tv)_K, \qquad&&\forall \tv\in\pkv{k-1}{K},\\
    \lef\langle(\piv\si)\nor-\tau_E\piwe \uv,\muv\rig\rangle_F &= \lef\langle\si\nor-\tau_E \pemv\uv,\muv\rig\rangle_F,\qquad&&\forall\muv\in\pkv{k}{F},
\end{alignat*}
for all faces $F$ of the element $K\in\te$. Above, $\pemv$ is the $\bos{L}^2$ projection onto $F$. Analogously, if $(\si,\uv)\in\uline{\bos{H}}^{k+1}(K)\times\bos{H}^{k+1}(K)$, then
\begin{subequations}\label{eq:ProjectionEstimates2}
\begin{align}
    \lef\|\piv\si-\si\rig\|_K &\lesssim h_K^{k+1}\lef(|\si|_{\uline{\bos{H}}^{k+1}(K)}+|\uv|_{\bos{H}^{k+1}(K)}\rig),\\
    \lef\|\piwe \uv-\uv\rig\|_K&\lesssim h_K^{k+1}\lef(|\uv|_{\bos{H}^{k+1}(K)}+|\midiv\si|_{\uline{\bos{H}}^k(K)}\rig).
\end{align}
\end{subequations}
In addition, for each element $K\in\te$, we will denote by $\pia\g$ the $\uline{\bos{L}}^2(K)$-projection of $\g$ on $\aspace{K}$. Thus, if $\g\in\uline{\bos{H}}^{k+1}(K)$, then 
\[
\lef\|\pia\g-\g\rig\|_K\lesssim h_K^{k+1}\lef|\g\rig|_{\uline{\bos{H}}^{k+1}(K)}.
\]
Having defined the projections, we now define the \textit{projection errors} in each of the volume unknowns by
\begin{alignat*}{8}
\delsi:=\,&\si-\piv\si, \qquad& \delu:=\,& \uv-\piwe\uv, \qquad&&\delg:=\,& \g-\pia\g, \\
\delq:=\,& \q-\piwa\q, \qquad& \delv:=\,& v-\piw v. && &
\end{alignat*}
The following quantity will play a fundamental role in the error estimations:
\[
\Theta(\si,\uv,\g,\q,v):=\lef(\|\delsi\|_{\te}^2+\|\delu\|_{\te}^2+\|\delg\|_{\te}^2+\|\delq\|_{\ta}^2+\|\delv\|_{\ta}^2\rig)^{1/2}.
\]
The next lemma follows readily from the projection bounds \eqref{eq:ProjectionEstimates1} and \eqref{eq:ProjectionEstimates2}.
\begin{lemma}\label{theta_bound}
    If $(\si,\uv,\g,\q,v)\in\uline{\bos{H}}^{k+1}(\ome)\times\bos{H}^{k+1}(\ome)\times\uline{\bos{H}}^{k+1}(\ome)\times\bos{H}^{k+1}(\oma)\times H^{k+1}(\oma)$, then 
    \begin{align*}
    &\Theta(\si,\uv,\g,\q,v)\lesssim h^{k+1}\lef(|\si|_{\uline{\bos{H}}^{k+1}(\ome)}+|\uv|_{\bos{H}^{k+1}(\ome)}+|\g|_{\uline{\bos{H}}^{k+1}(\ome)}+|\q|_{\bos{H}^{k+1}(\oma)}+|v|_{H^{k+1}(\oma)}\rig).
    \end{align*}
\end{lemma}
%
\subsection{Error estimates.}
%
Let us define the \textit{projections of the errors} (not to be confused with the projection errors defined above):
\begin{alignat*}{10}
    \esi &:= \piv\ten{\sigma}-\sih,&\qquad \esig\nor &:= \pemv(\ten{\sigma}\nor)-\sihg \nor,\qquad&
    \eu &:= \piwe\uv-\uvh,\\  \eug &:= \pemv\uv-\uvhg,& 
    \eg &:= \pia\ten{\gamma}-\gh,&
    \eq &:= \piwa\q-\qh, \\ \eqg\cdot\nor &:= \pem(\q\cdot \nor)-\q_h\cdot \nor,&
    \ev &:= \piw v -\vh,&\evg &:= \pem v-\vhg.
\end{alignat*}
Direct calculations imply that, for all $(\taut,\tv,\etat,\muv,\rv,w,\xi) \in \vhspace\times\whvespace\times\ahspace\times\mhvspace\times\whvaspace\times \whspace\times \mhspace$, the projections of the errors satisfy the following system:
\begin{subequations}
    \begin{align}
\ipte{\hookei\esi}{\taut}+\ipte{\eu}{\midiv \taut}+\ipte{\eg}{\taut}-\ipbe{\eug}{\taut\nor}&=-\ipte{\hookei\delsi}{\taut}-\ipte{\delg}{\taut},\label{error_h}\\[1ex]
\ipte{\esi}{\nabla \tv}-\ipbe{\esig\nor}{\tv} +\rE s^2\ipte{\eu}{\tv}& = -\rE s^2\ipte{\delu}{\tv},\label{error_i}\\[1ex]
\ipte{\esi}{\etat} & = -\ipte{\delsi}{\etat},\label{error_j}\\[1ex]
\ipfe{\esig\nor}{\muv} & = 0,\label{error_k}\\[1ex]
\ipta{\eq}{\rv}+\ipta{\ev}{\midiv \rv}-\ipba{\evg}{\rv \cdot\nor}&=-\ipta{\delq}{\rv},\label{error_HDG:acus_1}\\[1ex]
\ipta{\eq}{\nabla w}-\ipba{\eqg\cdot\nor}{w} +(s/c)^2\ipta{\ev}{w}& = -(s/c)^2\ipta{\delv}{w},\label{error_HDG:acus_2}\\[1ex]
\langle \eqg\cdot\nor,\xi\rangle_{\dta \setminus(\ga \cup \gaadir)} & = 0 ,\label{error_n}\\[1ex]
\langle \evg,\xi\rangle_{\gaadir} & = 0,\label{error_o}\\[1ex]
\label{error_Equation1E_h}
\ipg{\eqg \cdot \nora + s\, \eug\cdot \nore}{\xi} & = 0,\\[1ex]
\label{error_Equation1F_h}
\ipg{-\esig \nore - \rA s\, \evg \,\nora}{\muv} & = 0
\end{align}\label{error_hdgscheme}
while $\esig$ and $\eqg$ satisfy
\begin{alignat}{6}
    \esig\nor &= \esi\nor-\tau_E(\eu-\eug) \quad &\textrm{on}\,\,\partial \mathcal{T}_E,\label{error_Equation_flux1}\\
    \eqg\cdot\nor &= \eq\cdot\nor-\tau_A(\ev-\evg) \quad &\textrm{on}\,\,\partial \mathcal{T}_A\label{error_Equation_flux2}.
\end{alignat}
\end{subequations}
The following lemma can be proven by arguing as in the first part of the proof of \autoref{ExandUniq}.
\begin{lemma}
    The projections of the errors satisfy
    \begin{align}
    \nonumber
    e_E^2 + e_A^2 + \rE|s|^2 \re{s} & \nte{\eu}^2+\frac{\rA}{c^2}|s|^2\re{s}\nta{\ev}^2\\
    \nonumber
    =\, & -\re{s\ipte{\hookei\delsi}{\esi}}+\re{s\ipte{\eg}{\delsi}}-\re{s\ipte{\delg}{\esi}}\\
    \label{eq:Lem3}
    &-\rho_E|s|^2\re{s\ipte{\eu}{\delu}}-\rho_f\re{\conj{s}\ipta{\eq}{\delq}}-\dfrac{\rho_f}{c^2}|s|^2\re{s\ipta{\delv}{\ev}},
    \end{align}
  where
  \begin{align*}
       e_E &:= \sqrt{\ntehookei{\re{s}^{1/2}\esi}^2 + \nbte{\re{s}^{1/2}\tau_E^{1/2}(\eu-\eug)}^2},\\[2ex]
      e_A &:= \sqrt{ \nta{\rA^{1/2}\re{s}^{1/2}\eq}^2+\nbta{\rA^{1/2}\re{s}^{1/2}\tau_A^{1/2}(\ev-\evg)}^2}.
   \end{align*}
\end{lemma}
Let us now decompose now $\eg = \eg^0+\eg^c$, where $\eg^c$ is such that $\displaystyle \eg^c\vert_K=\frac{1}{|K|}\int_K \eg$ for all $K\in \te$ and $\eg^0:=\eg-\eg^c$. Since $\ten{\delta_{\si}}$ is orthogonal to piecewise constant polynomials, we have
\begin{align*}
     \ipte{\eg}{\ten{\delta_{\si}}} = \ipte{\eg^0}{\ten{\delta_{\si}}} +\ipte{\eg^c}{\ten{\delta_{\si}}} = \ipte{\eg^0}{\ten{\delta_{\si}}}.   \end{align*}
Then, using this information and  applying the triangle, Cauchy-Schwarz and Young inequalities several times to the expression \eqref{eq:Lem3}, we deduce that there exists a positive constant $C_1$, independent of $h$, such that

    \begin{equation}\label{mainresult}
    e_E^2 + e_A^2 + \rE|s|^2\re{s}\nte{\eu}^2+\frac{\rA}{c^2}|s|^2\re{s}\nta{\ev}^2\leq C_1\Theta(\si,\uv,\g,\q,v)^2+\dfrac{1}{2}\|\eg^0\|_{\te}^2.
\end{equation}

   \subsection{Error estimates for the rotation.}
It remains to obtain error bounds for $\eg^0$ and $\eg^c$. For the elasticity boundary value problem, these bounds were obtained in \cite{Ke}. In our case, we obtain an optimal error  estimate for $ \|\eg^0\|_{\te}$ following the same arguments presented in \cite{Ke}. However, the error estimate for the $L^2$-norm of $\eg^c$ depends on the term $\eug$ associated with the transmission conditions in $\Gamma$ as we will see in the next result.

\begin{lemma}\label{lema:e_rho}
There exist positive constants, $C_\gamma^c$ and $C_\gamma^0$, independent of $h$, such that
\begin{subequations}
\begin{align}\label{bound:e_rho_0}
 \|\eg^0\|_{\te} \leq C_\gamma^0\Theta(\si,\uv,\g,\q,v).  
\end{align}
and
\begin{align}\label{bound:e_rho_c}
\|\eg^c\|_{\te} 
\leq& C_\gamma^c\left(h^{-1/2}\Theta(\si,\uv,\g,\q,v)+ \|h^{-1}\eu\|_{\mathcal{T}_{A}}\right). \end{align}
\end{subequations}
\end{lemma}

\begin{proof}
By \cite[Theorem 3.6]{Ke}, we now that
\begin{align*}
 \|\eg^0\|_{\te} \leq \|\esi\|_{\te}+\|\delsi\|_{\te} +\|\delg\|_{\te}.  
\end{align*}
The first term can be bounded by \eqref{mainresult}, and therefore
\begin{align*}
 \|\eg^0\|_{\te}^2 \lesssim \Theta(\si,\uv,\g,\q,v)^2+\dfrac{s^2}{2}\|\ten{\delta_{\si}}\|_{\te}^2+\dfrac{1}{2}\|\eg^0\|_{\te}^2+\|\delsi\|_{\te}^2 +\|\delg\|_{\te}^2
 \lesssim  \Theta(\si,\uv,\g,\q,v)^2+\dfrac{1}{2}\|\eg^0\|_{\te}^2
\end{align*}
and \eqref{bound:e_rho_0} follows.

Now, we will modify the proof of \cite[Theorem 3.8]{Ke} to estimate $\eg^c$. Let $\etat:=\eg^c\in \ahcspace$. There exists $\ten{v}\in\uline{\bos{H}}(\text{div};\ome)\cap\pkt{1}{\te}$ satisfying the properties in \eqref{eq:temp3}. Then, taking $\taut=\ten{v}$ in \eqref{error_h} and using the fact that $\midiv \ten{v}=0$, we obtain
   \begin{align*}
\ipte{\hookei\esi}{\ten{v}}+\ipte{\eg}{\ten{v}}-\ipbe{\eug}{\ten{v}\nor}&=-\ipte{\hookei\delsi}{\ten{v}}-\ipte{\delg}{\ten{v}}.
 \end{align*}
 Now, since $\ipte{\eg}{\ten{v}}=\ipte{\eg^0}{\ten{v}}+\ipte{\eg^c}{\ten{v}}$ and $\ipte{\eg^c}{\ten{v}}=\|\eg^c\|_{\te}^2$, according to the second property in \eqref{eq:temp3}, we deduce that
   \begin{align*}
\|\eg^c\|_{\te}^2 =&-\ipte{\hookei\esi}{\ten{v}}-\ipte{\eg^0}{\ten{v}}-\ipte{\hookei\delsi}{\ten{v}}-\ipte{\delg}{\ten{v}}+\ipbe{\eug}{\ten{v}\nor}\\
\lesssim& \,\Theta(\si,\uv,\g,\q,v)\|\eg^c\|_{\te}+|\ipbe{\eug}{\ten{v}\nor}|,
 \end{align*}
where we have used the second and third properties of \eqref{eq:temp3}, and also the estimates \eqref{mainresult} and \eqref{bound:e_rho_0}. For the last term, we have
$
\ipbe{\eug}{\ten{v}\nor}  =  \langle \eug,\ten{v}\nor\rangle_\Gamma$ 
because $\eug$ is single-valued and $\ten{v}\in \ten{H}(\text{div};\ome)$, and this is precisely the term that in our case does not vanish, in contrast to the case in \cite{Ke}.

Let $e$ be a face in $\Gamma$ of an element $K\in\te$. By the discrete trace inequality, \eqref{mainresult} and \eqref{eq:temp3} we deduce that
\begin{align*}
 \langle \eug,\ten{v}\nor\rangle_e
\leq\,&  \langle \tau_E^{1/2}( \eug-\eu),\tau_E^{-1/2}\ten{v}\nor\rangle_e+\langle \eu,\ten{v}\nor\rangle_e\\ \lesssim\,& h^{-1/2}\Theta(\si,\uv,\g,\q,v)\|\eg^c\|_{K}+
\|h^{-1}\eu\|_K \|\eg^c\|_{K},
\end{align*}
which implies \eqref{bound:e_rho_c}.
\end{proof} 

If we consider the energy error estimate \eqref{mainresult} to bound $\nte{\eu}$, we will obtain the suboptimal error bound
\begin{align*}
\|\eg^c\|_{\te} 
\lesssim& 
h^{-1}\Theta(\si,\uv,\g,\q,v).\end{align*}

We can improve this result by considering a duality argument and gain an additional factor of $h^{1/2}$. In addition, the energy estimate \eqref{eq:Lem3} provides an order of convergence of $h^{k+1}$ for the projection errors $\eu$ and $e_v$. Using also a duality argument, it is possible to prove the superconvergence for $\eu$ and $e_v$, as we will show in the next section. 

\subsection{The duality argument.}
Given $\bs{\theta_e}\in\ldv{\ome}$ and $\theta_a\in\ld{\oma}$, we introduce the following auxiliary problem:
\begin{alignat*}{6}
\hookei\ten{\psi_e}-\nabla \bs{\phi_e}+ \bs{\xi_e}& =\ten{0} &\qquad & \text{ in }\ome ,\\
\midiv\ten{\psi_e} -\rE \conj{s^2}\,\bs{\phi_e}& =\bs{\theta_e} & \qquad &  \text{ in }  \ome,\\
\bs{\xi_a} + \nabla \phi_a & = \bos{0} & \qquad & \text{ in } \oma, \\
\midiv\bs{\xi_a}-\conj{(s/c)^2}\,\phi_a& =\theta_a\hspace{0.5cm} & \qquad & \text{ in } \oma,\\
\bs{\xi_a} \cdot \nora -  \conj{s}\, \bs{\phi_e}\cdot \nore & = 0 & \qquad & \text{ on }  \ga,\\
 \ten{\psi_e}  \nore - \rA \conj{s}\, \phi_a \,\nora & = 0 & \qquad & \text{ on }  \ga, \\
\phi_a & = 0 & \qquad & \text{ on }  \gaadir,\\
\bs{\xi_a}\cdot \nora & = 0 & \qquad & \text{ on }  \gaaneu.
\end{alignat*}
Here, $\ten{\xi_e}=\frac{1}{2}(\nabla\bs{\phi_e}-\nablat\bs{\phi_e})$. We assume that this problem admits the regularity estimate
\begin{equation}\label{additional_regularity}
    \|\ten{\psi_e}\|_{\bs{\underline{H}}^{s_e}(\ome)}+\|\bs{\phi_e}\|_{\bs{H}^{1+s_e}(\ome)}+\|\bs{\xi_a}\|_{\bs{H}^{s_a}(\oma)}+\|\phi_a\|_{H^{1+s_a}(\oma)}\lesssim \|\bs{\theta_e}\|_{\ome}+\|\theta_a\|_{\oma} 
\end{equation}
for some $s_e,s_a\geq 0$. 

Performing calculations analogous to those in \cite{CoGoSa2010,Ke}, it is possible to obtain the following lemma:
\begin{subequations}
    \begin{lemma}
    For any $\bs{\phi^k_e}\in\pkv{k}{\te}, \bs{\phi^{k-1}_e}\in\pkv{k-1}{\te}$ and $\bs{\theta_e}\in\ldv{\ome}$, we have
    \begin{align}
        \ipte{\eu}{\bs{\theta_e}} = \,\, &\ipte{\hookei\esi}{\ten{\delta_{\psi_e}}}+\ipte{\eg}{\ten{\delta_{\psi_e}}}+\ipte{\esi}{\bs{\delta_{\xi_e}}}\nonumber\\
        &+\ipte{\delsi}{\bs{\delta_{\xi_e}}}+\ipte{\hookei\delsi}{\ten{\delta_{\psi_e}}}+\ipte{\delg}{\ten{\delta_{\psi_e}}}-\ipte{\delsi}{\nabla(\bs{\phi_e}-\bs{\phi^k_e})}\nonumber\\
        & -\rE s^2\ipte{\uv-\uvh}{\bs{\delta_{\phi_e}}}+\rE s^2\ipte{\delu}{\bs{\phi_e}-\bs{\phi^{k-1}_e}}+\ipg{\eug}{\ten{\psi_e}\nore}-\ipg{\esig\nore}{\bs{\phi_e}}. \label{duality_e}
    \end{align}

    In addition, for any $\phi^k_a\in\pk{k}{\ta},\phi^{k-1}_a\in\pk{k-1}{\ta}$ and $\theta_a\in\ld{\oma}$, there holds
    \begin{align}
        \ipta{\ev}{\theta_a} = \,\,&\ipta{\q-\qh}{\bs{\delta_{\xi_a}}}-\ipta{\delq}{\nabla(\phi_a-\phi^k_a)}-\ipg{\eqg\cdot\nora}{\phi_a}\nonumber\\
        &-(s/c)^2\ipta{v-v_h}{\delta_{\phi_a}}+(s/c)^2\ipta{\delv}{\phi_a-\phi^{k-1}_a}+\ipg{\evg}{\bs{\xi_a}\cdot\nora}. \label{duality_a}
    \end{align}
\end{lemma}
\end{subequations}
 Based on the above two lemmas, we can derive the estimate. 
\begin{corollary}
    If the regularity assumption \eqref{additional_regularity} holds with $s_e,s_a\geq 0$ and $k\geq 1$, then
\begin{subequations}
\begin{align}\nta{\ev}+\nte{\eu}\lesssim\,& (h^{s_e}+h^{s_a})\Theta(\si,\uv,\g,\q,v),\label{bound:eu}\\[2ex]
\|\eg^c\|_{\te} 
\lesssim\,& (h^{-1/2}+h^{s_e-1}+h^{s_a-1})\Theta(\si,\uv,\g,\q,v). \label{bound2:e_rho_c}
\end{align}  

\end{subequations} 

\end{corollary}
\begin{proof}
    Taking $\bs{\theta_e}=\eu$ in \eqref{duality_e} and $\theta_a=\ev$ in \eqref{duality_a}, let us add $\nta{\ev}^2$ and $\rho_f^{-1}\nte{\eu}^2$. Then, by \eqref{error_Equation_flux1} and \eqref{error_Equation_flux2}, the terms in $\Gamma$ cancel out and we obtain
    \begin{align*}
        \nta{\ev}^2+\rho_f^{-1}\nte{\eu}^2 = \,\,& \ipta{\q-\qh}{\bs{\delta_{\xi_a}}}-\ipta{\delq}{\nabla(\phi_a-\phi^k_a)}-(s/c)^2\ipta{v-v_h}{\delta_{\phi_a}}\\
        & +(s/c)^2\ipta{\delv}{\phi_a-\phi^{k-1}_a}+\rho_f^{-1}\left\{\ipte{\hookei\esi}{\ten{\delta_{\psi_e}}}+\ipte{\eg}{\ten{\delta_{\psi_e}}}\right.\\& \left.+\ipte{\esi}{\bs{\delta_{\xi_e}}} +\ipte{\delsi}{\bs{\delta_{\xi_e}}}+\ipte{\hookei\delsi}{\ten{\delta_{\psi_e}}}+\ipte{\delg}{\ten{\delta_{\psi_e}}}\right.\\&\left.-\ipte{\delsi}{\nabla(\bs{\phi_e}-\bs{\phi^k_e})} -\rE s^2\ipte{\uv-\uvh}{\bs{\delta_{\phi_e}}}+\rE s^2\ipte{\delu}{\bs{\phi_e}-\bs{\phi^{k-1}_e}}\right\}
    \end{align*}
    Now, we notice that
    \begin{align*}
     \ipte{\eg}{\ten{\delta_{\psi_e}}} = \ipte{\eg^0}{\ten{\delta_{\psi_e}}} +\ipte{\eg^c}{\ten{\delta_{\psi_e}}} = \ipte{\eg^0}{\ten{\delta_{\psi_e}}},   
    \end{align*}
    because $\ten{\delta_{\psi_e}}$ is orthogonal to piecewise constant polynomials. In addition, the terms on $\Gamma$ cancel each other out. Then, applying the triangular and Cauchy-Schwarz inequalities, we obtain 
    \begin{align}\label{aux:duality1}
        \nta{\ev}^2+\rho_f\nte{\eu}^2\lesssim   (OPT \times APT)+ \|\ten{\delta_{\psi_e}}\|_{\te}\|\eg^0\|_{\te},
    \end{align}
    where OPT stands for ``original problem terms'' and APT for ``auxiliary problem terms'':
    \begin{align*}
        OPT := \bigg(\,\,&\nta{\q-\qh}^2+\nta{v-v_h}^2+\nte{\uv-\uvh}^2+\nte{\esi}^2\\
        &+\nta{\delq}^2+\nta{\delv}^2+\nte{\delsi}^2+\nte{\delu}^2+\nte{\delg}^2\bigg)^{1/2}
    \end{align*}
    and
    \begin{align*}
        APT := \bigg(\,\,&\nta{\bs{\delta_{\xi_a}}}^2+\nta{\delta_{\phi_a}}^2+\nta{\nabla(\phi_a-\phi^k_a)}^2+\nta{\phi_a-\phi^{k-1}_a}^2\\
        &
        +\nte{\ten{\delta_{\xi_e}}}^2+\nte{\bs{\delta_{\phi_e}}}^2+\nte{\nabla(\bs{\phi_e}-\bs{\phi^k_e})}^2+\nte{\bs{\phi_e}-\bs{\phi^{k-1}_e}}^2\bigg)^{1/2}.
    \end{align*}

In the OPT term, we add and subtract the projections $\piwa\q$, $\piw v$ and $\piwe\uv$ in the first three terms, use \eqref{mainresult}  and the definition of $\Theta(\si,\uv,\g,\q,v)$, to conclude that 
    \begin{equation*}
        OPT \lesssim \Theta(\si,\uv,\g,\q,v)+\sqrt{\dfrac{1}{2}}\|\eg^0\|_{\te}.
    \end{equation*}

    Regarding the APT, we first consider $\phi_a^{k-1}$ and $\phi_a^{k}$ as the $L^2$-projections of $\phi_a$ over $\pk{k-1}{\ta}$ and $\pk{k}{\ta}$, resp. Similarly, we take 
    $\boldsymbol{\phi_e^{k-1}}$ and  $\boldsymbol{\phi_e^{k}}$ as the $L^2$-projections of $\boldsymbol{\phi_e}$ over $\pkv{k-1}{\te}$ and $\pkv{k}{\te}$, resp.
    Then, by the approximation properties of the $L^2$- \cite[Lemma 1.58]{ErnDiPietro2012} and
  the HDG-projections \eqref{eq:ProjectionEstimates1}-\eqref{eq:ProjectionEstimates2}, and assuming the regularity assumption \eqref{additional_regularity}, we can deduce  that 
    \begin{equation*}
        APT \lesssim (h^{s_e}+h^{s_a})\left(\nta{\ev}+\sqrt{\rho_f}\nte{\eu}\right).
    \end{equation*}
    Then, replacing these expressions in \eqref{aux:duality1}, and noticing that \[\|\ten{\delta_{\psi_e}}\|_{\te}\lesssim h (\|\bs{\theta_e}\|_{\ome}+\|\theta_a\|_{\oma})\lesssim h \left(\nta{\ev}+\sqrt{\rho_f}\nte{\eu}\right),\] we obtain that
    \begin{align*}
\nta{\ev}^2+\rho_f\nte{\eu}^2\lesssim&   \left(\Theta(\si,\uv,\g,\q,v)+\sqrt{\dfrac{1}{2}}\|\eg^0\|_{\te}\right) \left((h^{s_e}+h^{s_a})\left(\nta{\ev}+\sqrt{\rho_f}\nte{\eu}\right)\right)\\
&+ \|\ten{\delta_{\psi_e}}\|_{\te}\|\eg^0\|_{\te}\\
        \lesssim& \left(\Theta(\si,\uv,\g,\q,v)+\|\eg^0\|_{\te}\right) \left((h^{s_e}+h^{s_a})\left(\nta{\ev}+\sqrt{\rho_f}\nte{\eu}\right)\right),
    \end{align*}
    which implies that
\begin{align*}      \nta{\ev}^2+\rho_f\nte{\eu}^2\lesssim&\,(h^{s_e}+h^{s_a})\,   \left(\Theta(\si,\uv,\g,\q,v)+\|\eg^0\|_{\te}\right)\lesssim\,(h^{s_e}+h^{s_a})\,\Theta(\si,\uv,\g,\q,v),
    \end{align*}
\end{proof}

Summarizing all previous estimates, and using the estimate in Lemma \ref{theta_bound}, we have the following result.

\begin{theorem}\label{thm:error_estimates}
If $(\si,\uv,\g,\q,v)\in\uline{\bos{H}}^{k+1}(\ome)\times\bos{H}^{k+1}(\ome)\times\uline{\bos{H}}^{k+1}(\ome)\times\bos{H}^{k+1}(\oma)\times H^{k+1}(\oma)$ and $k\geq 1,$ then
\[
\|\si-\sih\|_{\te}+ \|\uv-\uvh\|_{\te}+
    \|\q-\qh\|_{\ta}+ \|v-v_h\|_{\ta}\lesssim h^{k+1}.
\]
and
\[
\|\g-\gh\|_{\ta}\lesssim h^{k}(h^{1/2}+h^{s_e}+h^{s_a})\,,
\]
where $s_a$ and $s_e$ are the coefficients of the Sobolev regularity for the dual problem appearing in assumption \eqref{additional_regularity}.
\end{theorem}

Finally, we have the following error estimates for the numerical traces:

\begin{theorem}\label{lem:error_traces}
Under the same hypothesis of the previous theorem, there holds
\begin{subequations}
\begin{equation}\label{bound:uhat}
\vertiii{\eug}_{\partial\mathcal{T}_E}\lesssim (h^{-1/2}+h^{s_e-1}+h^{s_a-1})h^{k+1}
\end{equation}
and
\begin{equation}\label{bound:vhat}
\vertiii{e_{\hat{v}}}_{\partial\mathcal{T}_A}\lesssim (h+h^{s_e}+h^{s_a})h^{k+1}.
\end{equation}
\end{subequations}
where, for $\dagger\in\{A,E\}$, we consider the norm
\[
\vertiii{\cdot}_{\partial\mathcal{T}_\dagger}:=\lef(\dis\sum_{K\in\mathcal{T}_\dagger} h_K\|\cdot\|_{\partial K}^2\rig)^{1/2}.
\]
Above, $s_a$ and $s_e$ are the Sobolev regularity coefficients of the dual problem appearing in assumption \eqref{additional_regularity}.
\end{theorem}
\begin{proof}
    By following the argument in the proof of Theorem 4.1 in \cite{CoGoSa2010}, let $K\in\te$ and $\taut_K\in  \pkt{k}{K}$ such that $ \taut\nor=\eug$ on $\partial K$ and $\|\taut_K\|_K\lesssim h_K^{1/2}\|\eug\|_{\partial K}$. According to \eqref{error_h}, taking $\taut=\taut_K$ in $K$ and $\taut=0$ otherwise, we can write
   \begin{align*}
\|\eug\|_{\partial K}^2=(\hookei\esi,\taut)_K+(\eu,\midiv \taut)_K+(\eg,\taut)_K+(\hookei\delsi,\taut)_K+(\delg,\taut)_K.
\end{align*}
By the Cauchy-Schwarz and inverse inequalities, we can obtain
   \begin{align*}
\|\eug\|_{\partial K}^2\lesssim \left(\|\esi\|_{K}+h_K^{-1}\|\eu\|_{K}+\|\eg\|_{K}+\|\delsi\|_{K}+\|\delg\|_{K}\right)\|\taut\|_K,
\end{align*}
which implies that
   \begin{align*}
\|\eug\|_{\partial K}\lesssim h_K^{1/2}\|\esi\|_{K}+h_K^{-1/2}\|\eu\|_{K}+h_K^{1/2}\|\eg\|_{K}+h_K^{1/2}\|\delsi\|_{K}+h_K^{1/2}\|\delg\|_{K},
\end{align*}
because $\|\taut_K\|_K\lesssim h_K^{1/2} \|\eug\|_{\partial K}$. This expression, together with \eqref{bound:eu} and \eqref{mainresult} to bound $\|\eu\|_{K}$ and $\|\esi\|_{K}$, respectively, implies that    \begin{align*}
\left(\sum_{K\in\mathcal{T}_E}h_K\|\eug\|_{\partial K}^2\right)^{1/2}\lesssim\,& h\Theta(\si,\uv,\g,\q,v)+\|\eg\|_{\mathcal{T}_E}.
\end{align*}
The estimate  follows after using \eqref{bound:e_rho_0} and \eqref{bound2:e_rho_c}. 

A similar procedure for $e_{\hat{v}}$ (see also the proof of Theorem 4.1 of \cite{CoGoSa2010}) leads to
  \begin{align*}
h_K^{1/2}\|e_{\hat{v}}\|_{\partial K}\lesssim& h_K\|\eq\|_{K}+\|e_v\|_{K}+h_K\|\delq\|_{K},
\end{align*}
for $K\in \mathcal{T}_A$.
Adding over $K$, 
  \begin{align*}
\left(\sum_{K\in \mathcal{T}_A} h_K\|e_{\hat{v}}\|_{\partial K}^2\right)^{1/2}
\lesssim &
(h+h^{s_e}+h^{s_a})\Theta(\si,\uv,\g,\q,v),
\end{align*}
where we used \eqref{mainresult}, \eqref{bound:e_rho_0} and
\eqref{bound:eu}. The result follows after considering the estimate in Lemma \ref{theta_bound}
 \end{proof}

\section{Numerical Experiments}
%
\subsection{A manufactured solution.}
%

We now test our HDG scheme applied to the coupled problem \eqref{Equation1A}-\eqref{Equation1H} with Dirichlet boundary conditions $v=g_D$ on $\gaa$. We take a manufactured acoustic field $v(x,y)=\sin(x)\sin(y)$. The source $f$ and boundary data $g_D$ are set in such a way that $v$ satisfies \eqref{Equation1C}-\eqref{Equation1D} 
in a domain $\Omega_A=(-2,2)^2$, with $c=1$ and $s=2-i$. For the elastic region, we consider $\ome=(-1,1)^2, \rho_E=1$ and $\tau_E = 1$ everywhere. The source $\efe$ is defined such that
\[
\boldsymbol{u}(x,y)=\begin{pmatrix}
    \sin(\pi x)\cos(\pi y)\\
    \cos(\pi x)\sin(\pi y)
\end{pmatrix}, \quad (x,y)\in (-1,1)^2,
\]
satisfies \eqref{Equation1A}-\eqref{Equation1B}. We set the field $\vinc(x,y)=-\sin(x)\sin(y)$ and include additional terms on the right-hand sides of \eqref{Equation1E}-\eqref{Equation1F} so that our manufactured solution satisfies them.

\autoref{fig:plots_coupled} presents the numerical results obtained. The experimental orders of convergence of $\si$,  $\uv$, $\q$ and $v$ coincide with the theoretical results predicted by \autoref{thm:error_estimates}.
Now, for the rotation, \autoref{thm:error_estimates} guarantees an order $h^{k+\min\{1/2,s_e,s_1\}}$, where we recall that $s_e$ and $s_a$ are the regularity indices in \eqref{additional_regularity}. Numerically, we observe a better result and obtain a convergence rate of $h^{k+1}$. Moreover, \autoref{lem:error_traces} predicts $\vertiii{\eug}_{\partial\mathcal{T}_E}\lesssim h^{k+\min\{1/2,s_e,s_a}\}$ and $\vertiii{e_{\hat{v}}}_{\partial\mathcal{T}_A}\lesssim h^{k+1+\min\{1,s_e,s_a\}}$. Computationally, convergence of order $k+2$ is observed for the numerical traces.

\begin{figure}
\centering
\begin{tabular}{ccc}
 & \normalsize{\textbf{Acoustic variables} } & \\[1ex]
 Degree $k=1$ &  Degree $k=2$ & Degree $k=3$\\[1ex]
\!\!\includegraphics[width= 0.32\linewidth]{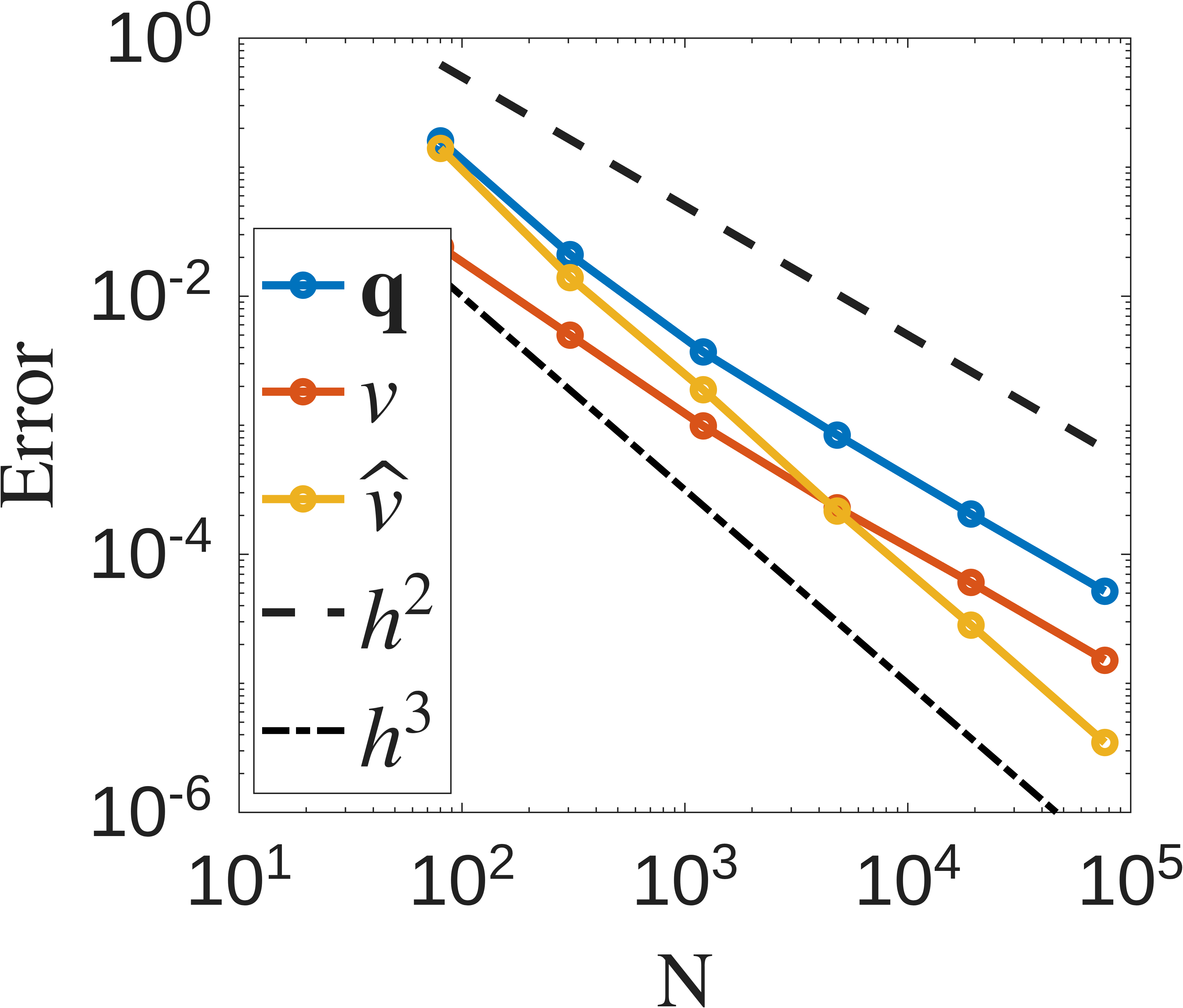} &  
\!\!\includegraphics[width= 0.32\linewidth]{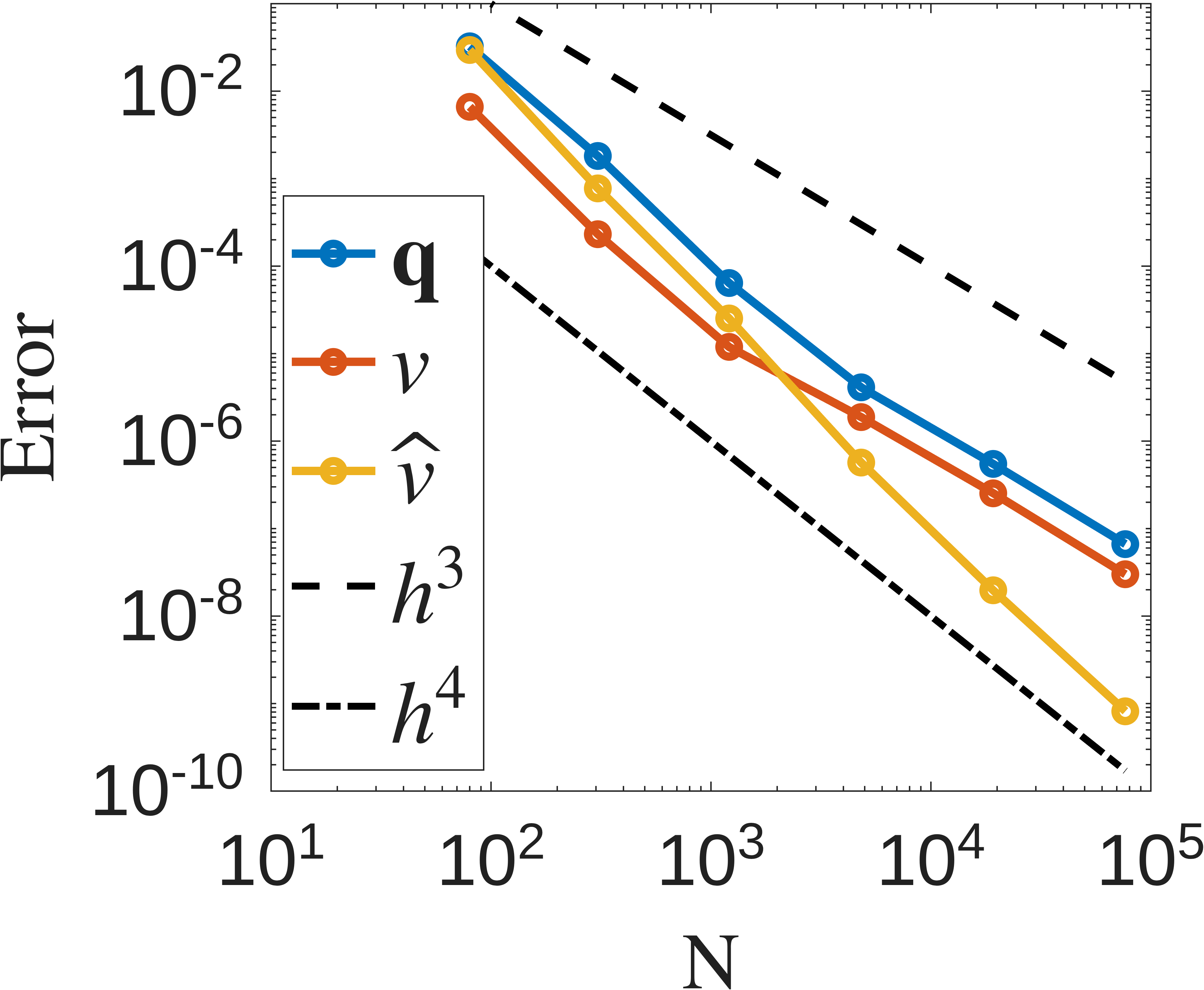} &
\!\!\includegraphics[width= 0.32\linewidth]{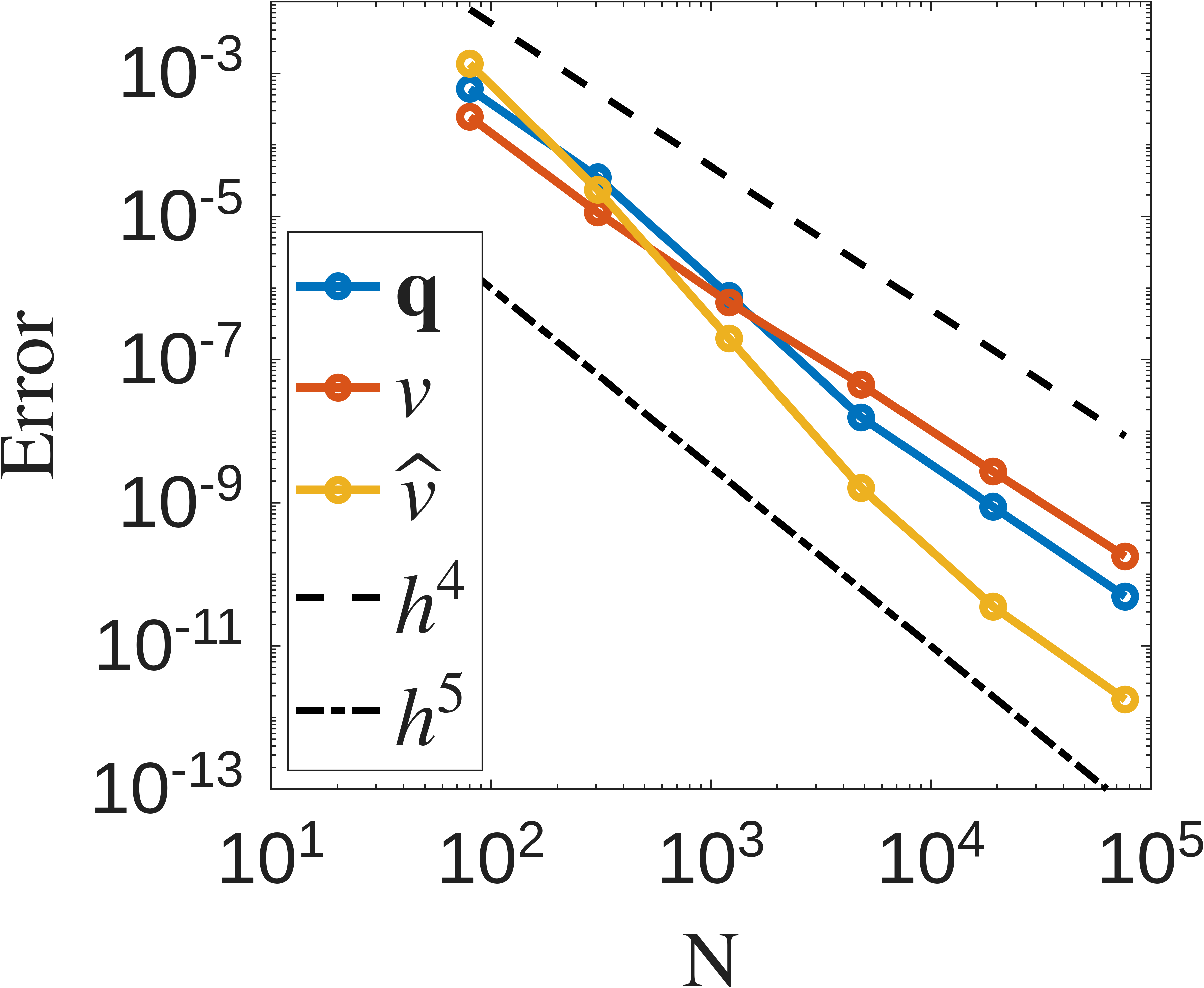}\\[1ex]
 & \normalsize{\textbf{Elastic Variables}} & \\
 Degree $k=1$ &  Degree $k=2$ & Degree $k=3$\\[1ex]
\!\!\includegraphics[width= 0.32\linewidth]{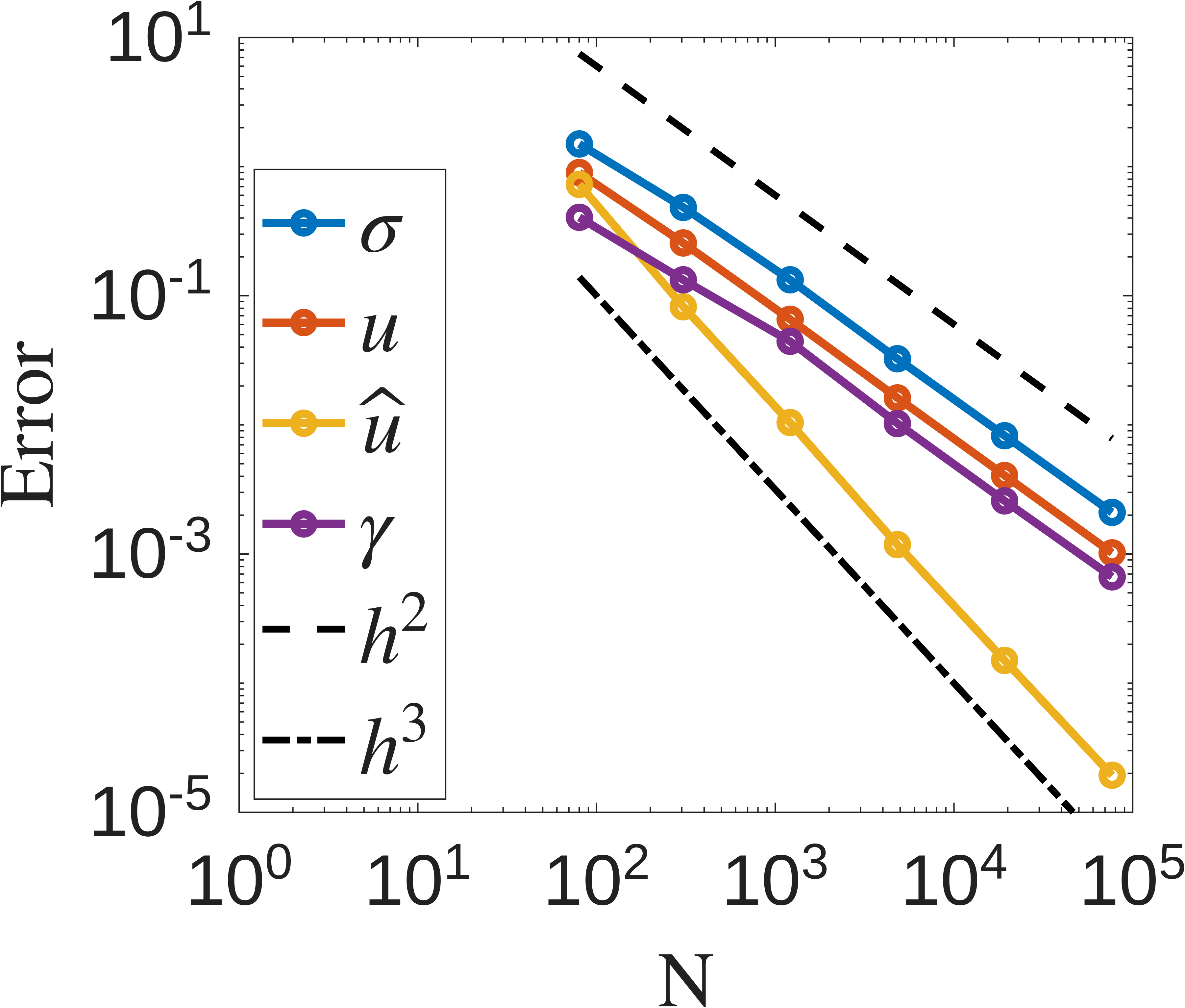} &
\!\!\includegraphics[width= 0.32\linewidth]{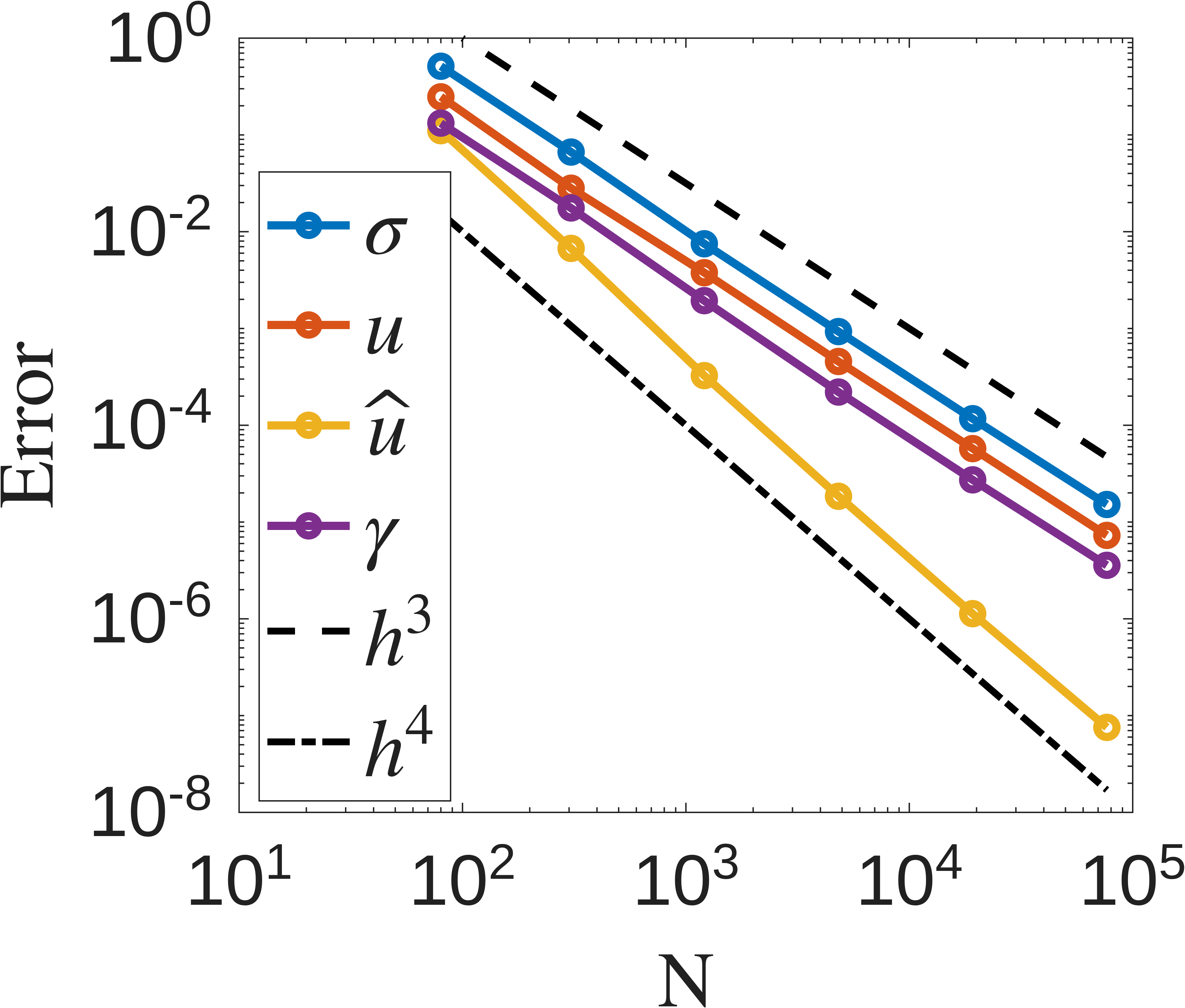} &
\!\!\includegraphics[width= 0.32\linewidth]{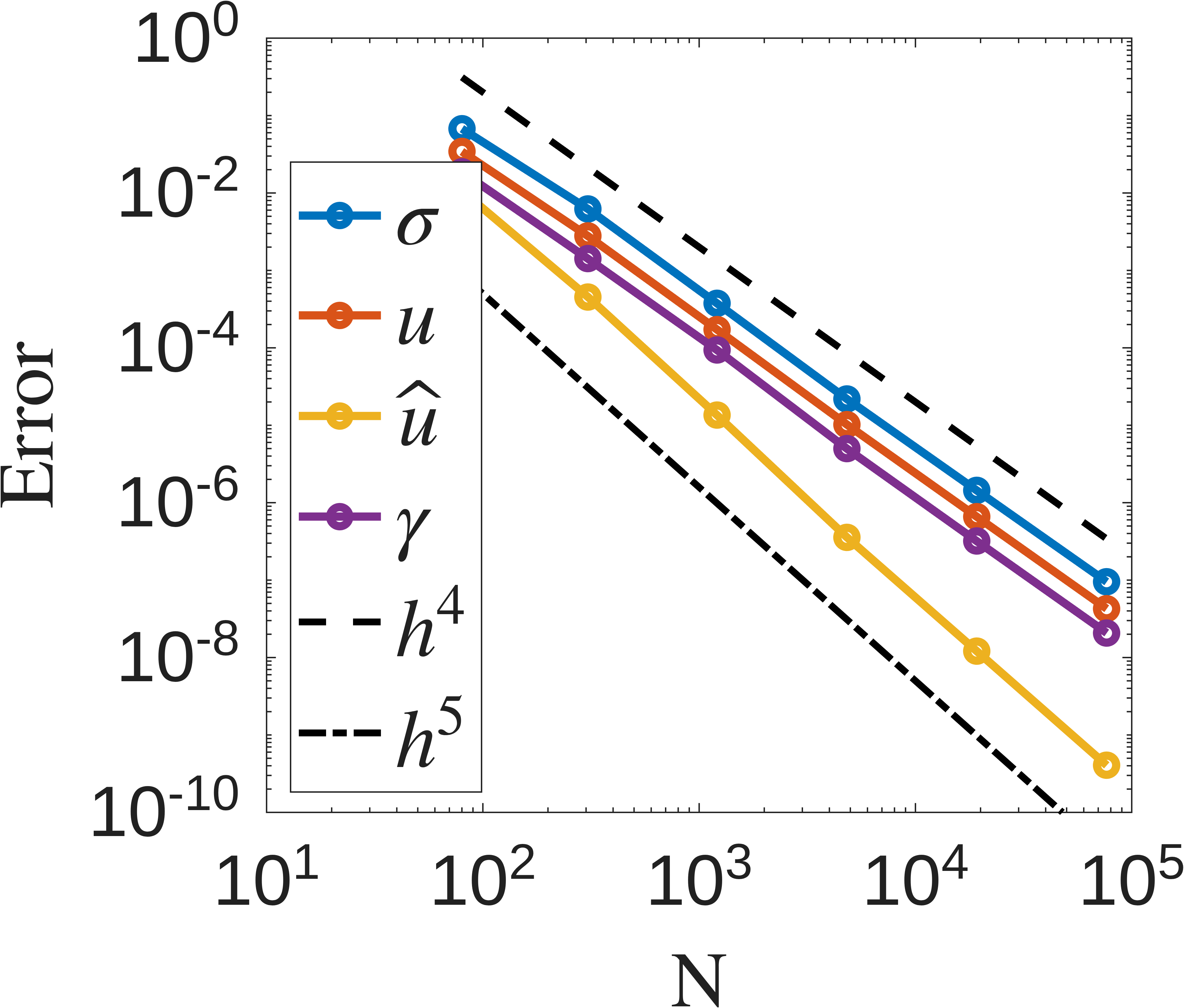}
\end{tabular}
\caption{Discretization error as a function of the number of elements in the domain for the coupled acoustic/elastic problem. Acoustic variables are displayed on the first column and elastic variables on the second column.}
\label{fig:plots_coupled}
\end{figure}

\subsection{A test case.}
We now simulate a test case in two concentric annular regions considered in \cite{BaDjEs014} and \cite{HuKaMo2008}. Let
\[
\Omega_E=\{(x,y)\in \mathbb{R}^2: x^2+y^2<a\} \qquad \text{ and } \qquad \Omega_A=\{(x,y)\in \mathbb{R}^2: a<x^2+y^2<b\}
\]
with $a=10$mm and $b=20$mm. The elastic structure enclosed by $\Omega_E$ is aluminum with density $\rho_E=2700$kg/m$^3$, Young's modulus $E=70\times 10^9$Pa and Poisson ratio $\nu=0.33$. The surrounding annulus $\Omega_A$ is considered to be filled with water with density $\rho_f=1000$kg/m$^3$, speed of sound $c = 1500$ m/s and wave number $\kappa=\omega/c$. We look for the displacement field $\boldsymbol{u}$ and the acoustic scattered pressure $p$ that satisfy the following time-harmonic equation with angular frequency $\omega$:
\begin{subequations}\label{model-problem}
\begin{align}
    -\midiv\boldsymbol{\sigma}-\rE \omega^2\uv &= \boldsymbol{0} &\quad&\text{in }\ome,\\
    -\Delta p-\kappa^2p&=0 &\quad&\text{in }\oma,\\
    -\nabla  p^\text{tot}\cdot\nore+\omega^2\rA \,\uv\cdot\nore&=0 &\quad&\text{on }\ga,\\
p^{\text{tot}}\nore+\boldsymbol{\sigma}\nore &= \bos{0} &\quad&\text{on }\ga,\\
   \nabla p \cdot \nore-i \kappa p &=0 &\quad&\text{on } \partial \Omega_A,
\end{align}
 
\end{subequations}
where $p^{\text{tot}}=p+p^{\text{inc}}$ and $p^{\text{inc}}(x,y)=e^{i\,\kappa\,x\,}$. The problem can be written in velocity-displacement formulation (i.e. in the form of system \eqref{Equation1}) by recalling that the pressure is given in terms of the velocity potential by $p=-\rho_f(i\omega) v = -\rho_fsv$. Instead of the Dirichlet boundary conditions of our model problem, we must account for the first-order absorbing boundary condition. To handle this scenario, we replace \eqref{seqo} by 
\[
\langle \qhg \cdot \boldsymbol{n}_A-i\kappa \widehat{v}_h,\xi \rangle_{\Gamma_A^D}=0.
\]
On the other hand, even though we considered the Laplace parameter $s\in \{z\in \C: \re{z}>0\}$, the results of this work can be easily extended to the case $\re{s}=0$.

To compare the solution obtained by the HDG scheme, we compute the solution to \eqref{model-problem} by a Fourier series representation (cf.\cite{BaDjEs014} and \cite{HuKaMo2008}) with 30 modes, depicted in the right panel of Figure \ref{fig:meshAndsolution}. The system  \eqref{model-problem} is then discretized by the HDG scheme described on a the mesh with 574 elements in the acoustic outer annulus $\Omega_A$, and 207 elements in the interior elastic circle $\Omega_E$. The geometry and the mesh are depicted in the left panel of Figure \ref{fig:meshAndsolution}.

\begin{figure}[ht!]
\centering
\begin{tabular}{ccc}
\quad Mesh level 1 & \quad  Mesh level 2  & \quad  Series solution \\[1ex] 
\qquad \includegraphics[height = 0.22\linewidth]{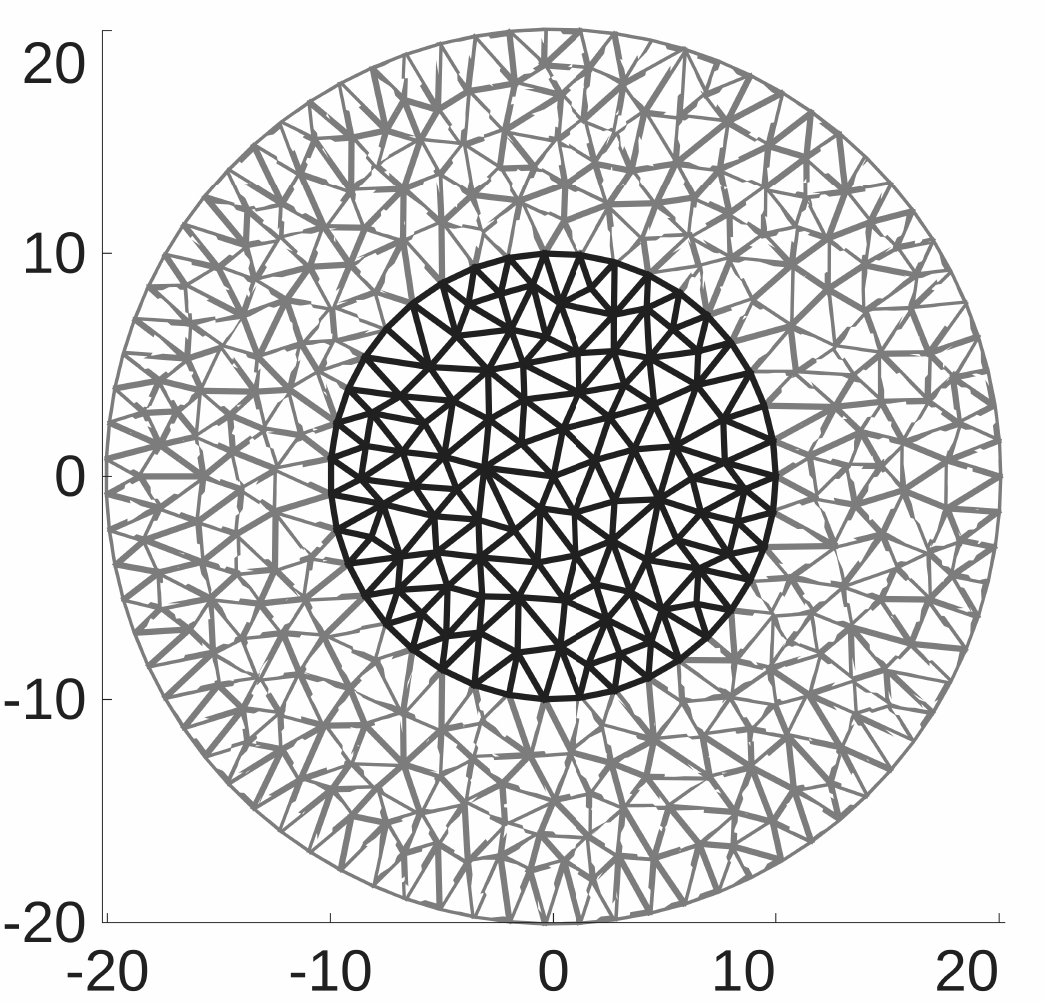} \qquad &
\qquad \includegraphics[height = 0.22\linewidth]{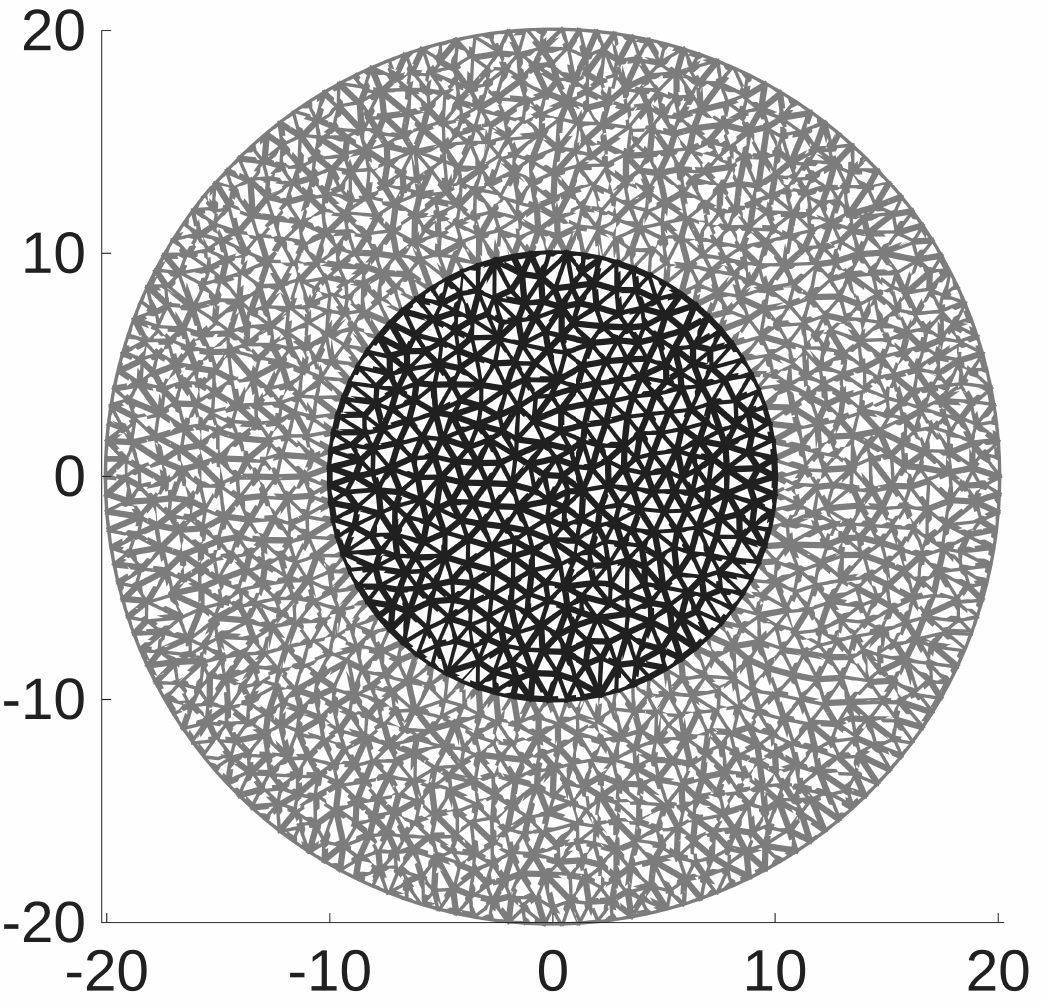} \qquad &
\qquad \includegraphics[height= 0.23\linewidth]{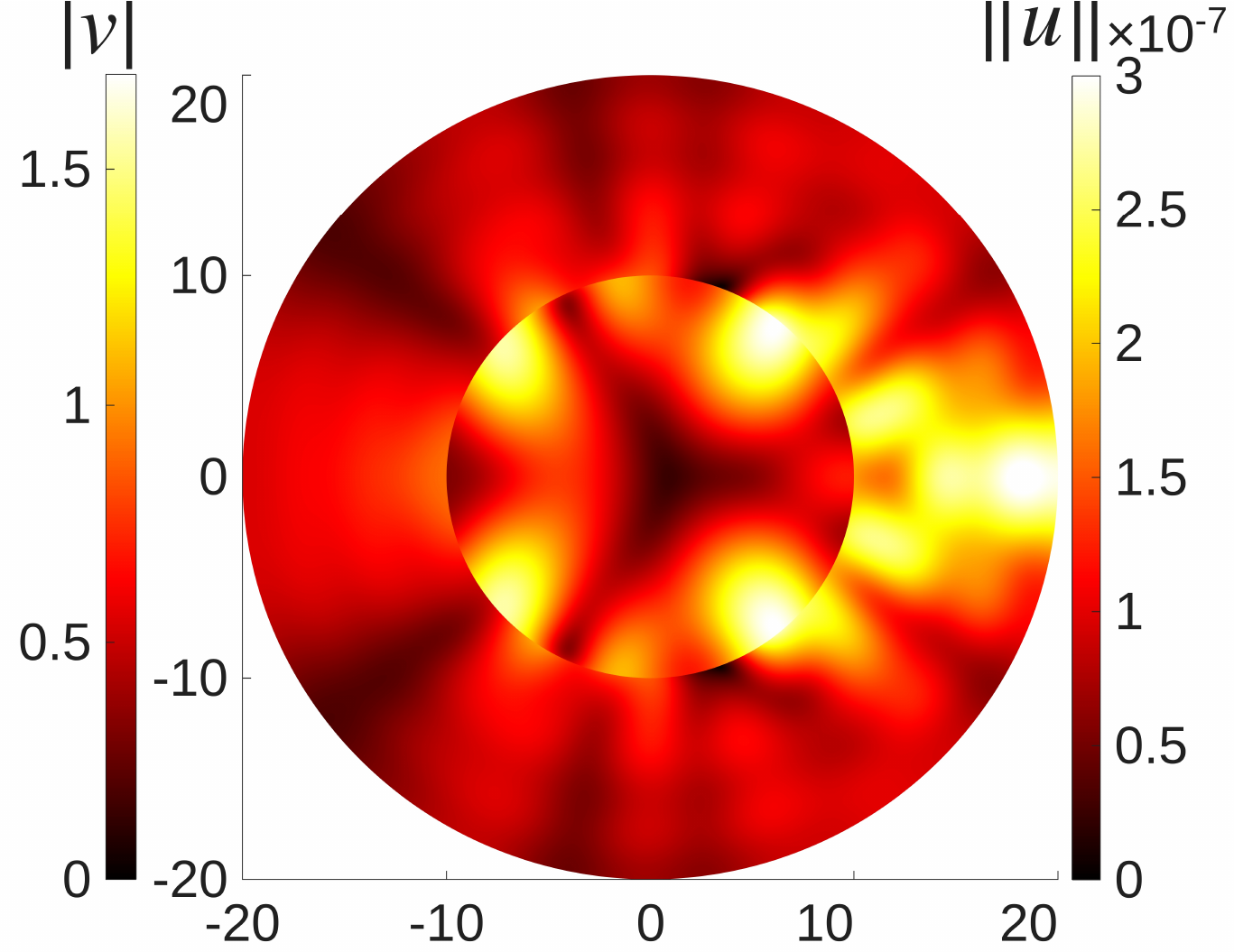}
\end{tabular}
\caption{The computational mesh is composed of two regions, the exterior annulus depicted in gray corresponds to the acoustic domain $\Omega_A$, while the interior circle in black is the elastic domain $\Omega_E$. Left: The mesh consists of 574 elements in the acoustic region, and 207 elements in the elastic region. Center: The mesh consists of 2341 elements in the acoustic region, and 763 elements in the elastic region. Right: The Fourier series approximation to the solution with 30 terms, the image depicts the magnitude of the acoustic scalar wave potential $v$ and the magnitude of the elastic displacement $\|\boldsymbol u\|$. Both are measured in millimeters. Note the difference in magnitudes, while the acoustic potential (left color bar) is of the order of 1 millimeter, the acoustic displacement has magnitude of the order $10^{-7}$ millimeters.} 
\label{fig:meshAndsolution}
\end{figure}

This problem is particularly  challenging due to the vastly different difference in wave propagation speed in the two media, which can lead to conditioning issues. We show the results of simulations showcasing $p$ and $h$ convergence. Figure \ref{fig:meshAndsolution} depicts solutions for polynomial degrees $k=1,2,3$ varying horizonally over columns, and for the meshes level 1 and 2 (depicted in Figure \ref{fig:plots_coupled2_mesh0}) varying vertically along rows.

\begin{figure}[ht!]
\centering
\begin{tabular}{cccc}
& \textbf{Polynomial degree $k=1$} & \textbf{ Polynomial degree $k=2$} & \textbf{Polynomial degree $k=3$}\\[2ex] 
\rotatebox{90}{\hspace{1.2cm}\textbf{Mesh level 1}} & \hspace{-.35cm}\includegraphics[width= 0.3\linewidth]{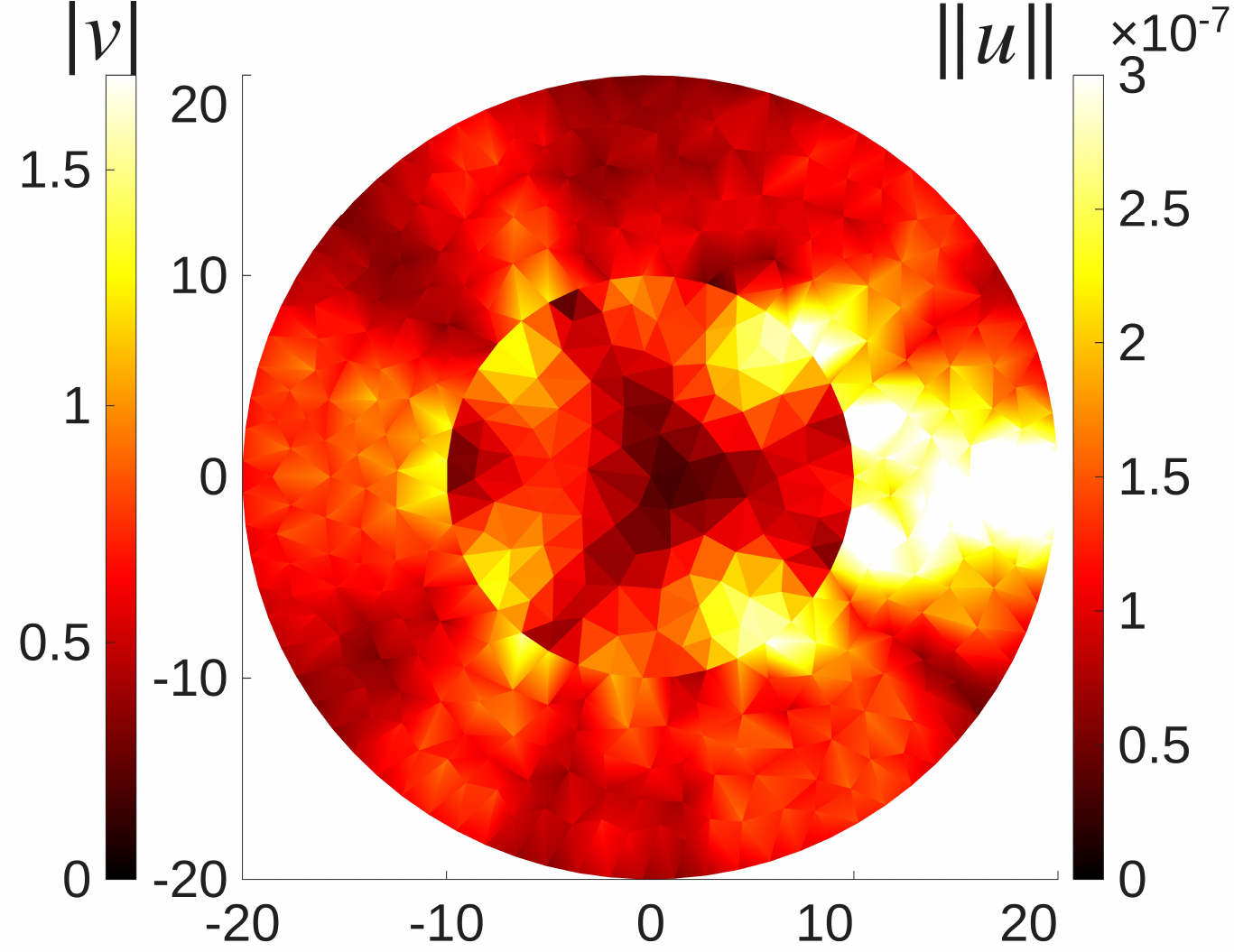} \qquad &
\includegraphics[width= 0.3\linewidth]{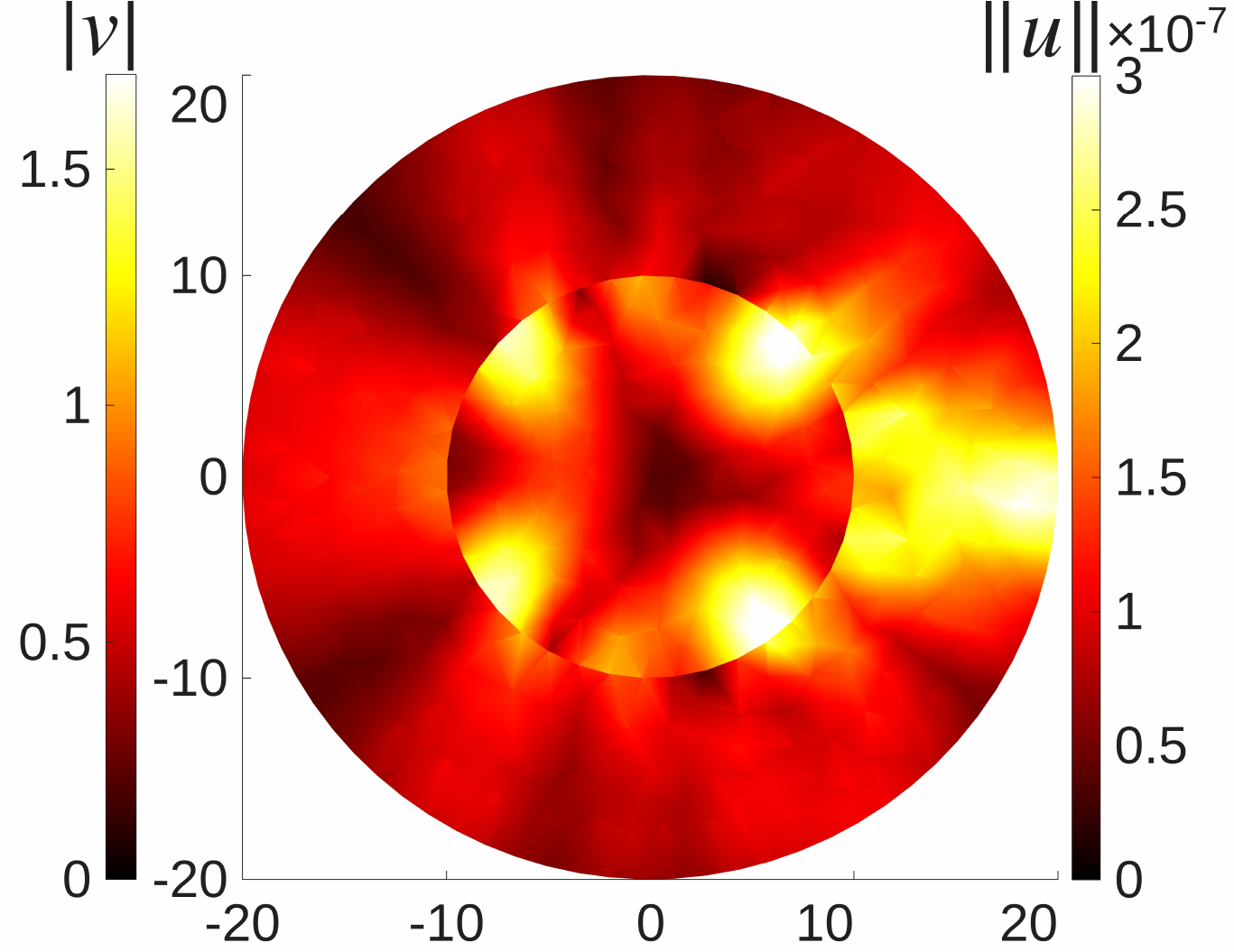} \qquad &
\includegraphics[width= 0.3\linewidth]{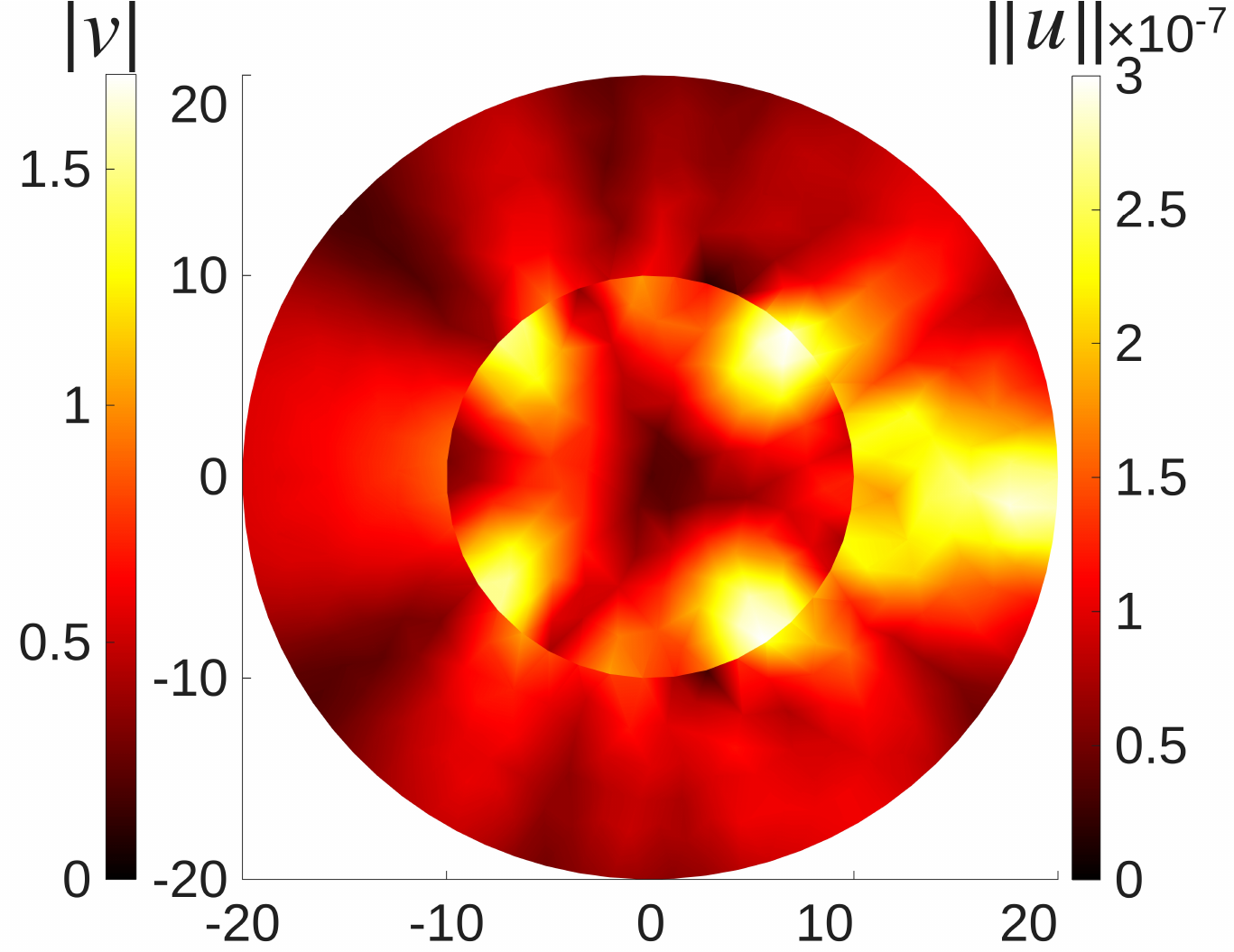}\\[4ex]
\rotatebox{90}{\hspace{1.2cm} \textbf{Mesh level 2}} &
\hspace{-.35cm}\includegraphics[width= 0.3\linewidth]{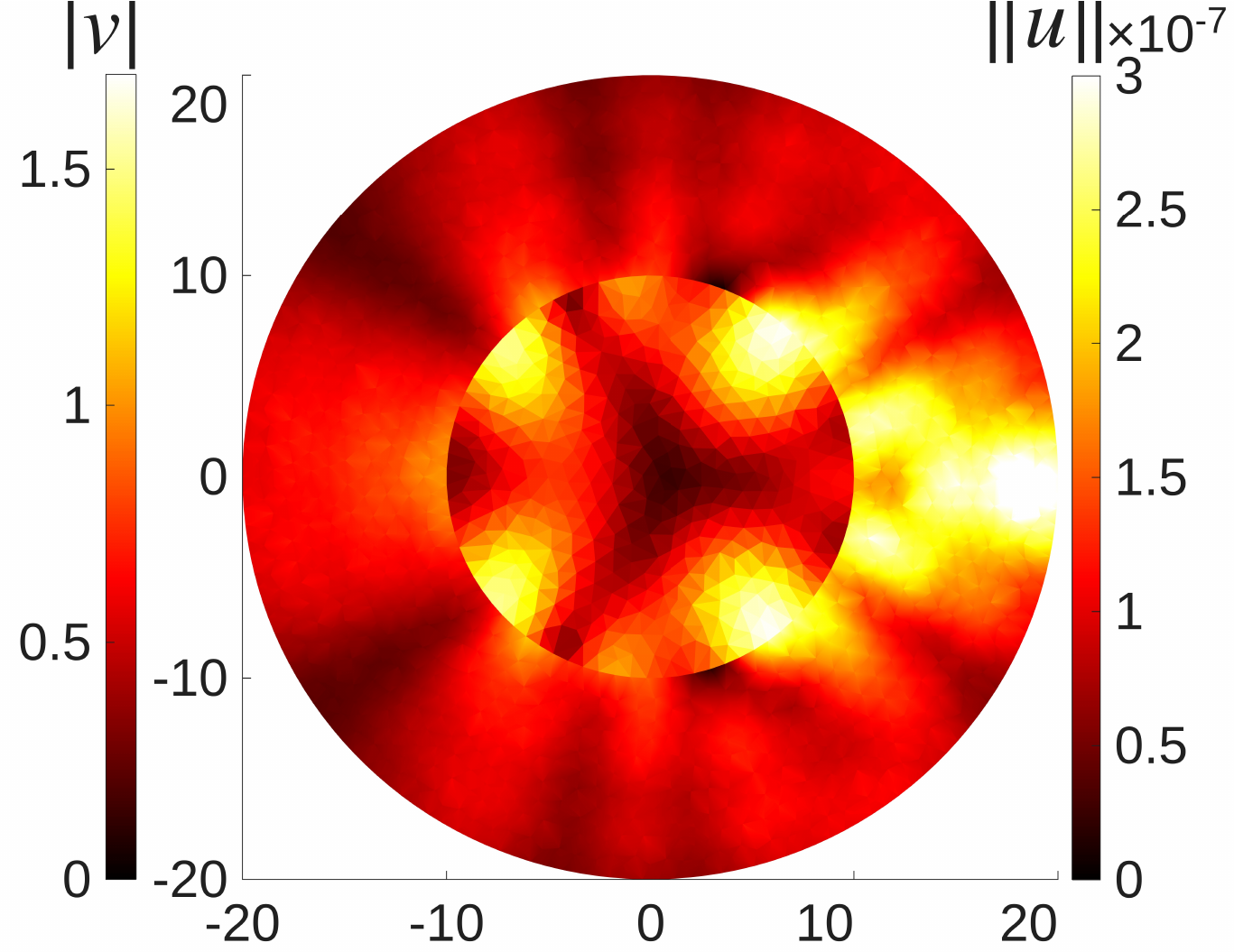} \qquad &
\includegraphics[width= 0.3\linewidth]{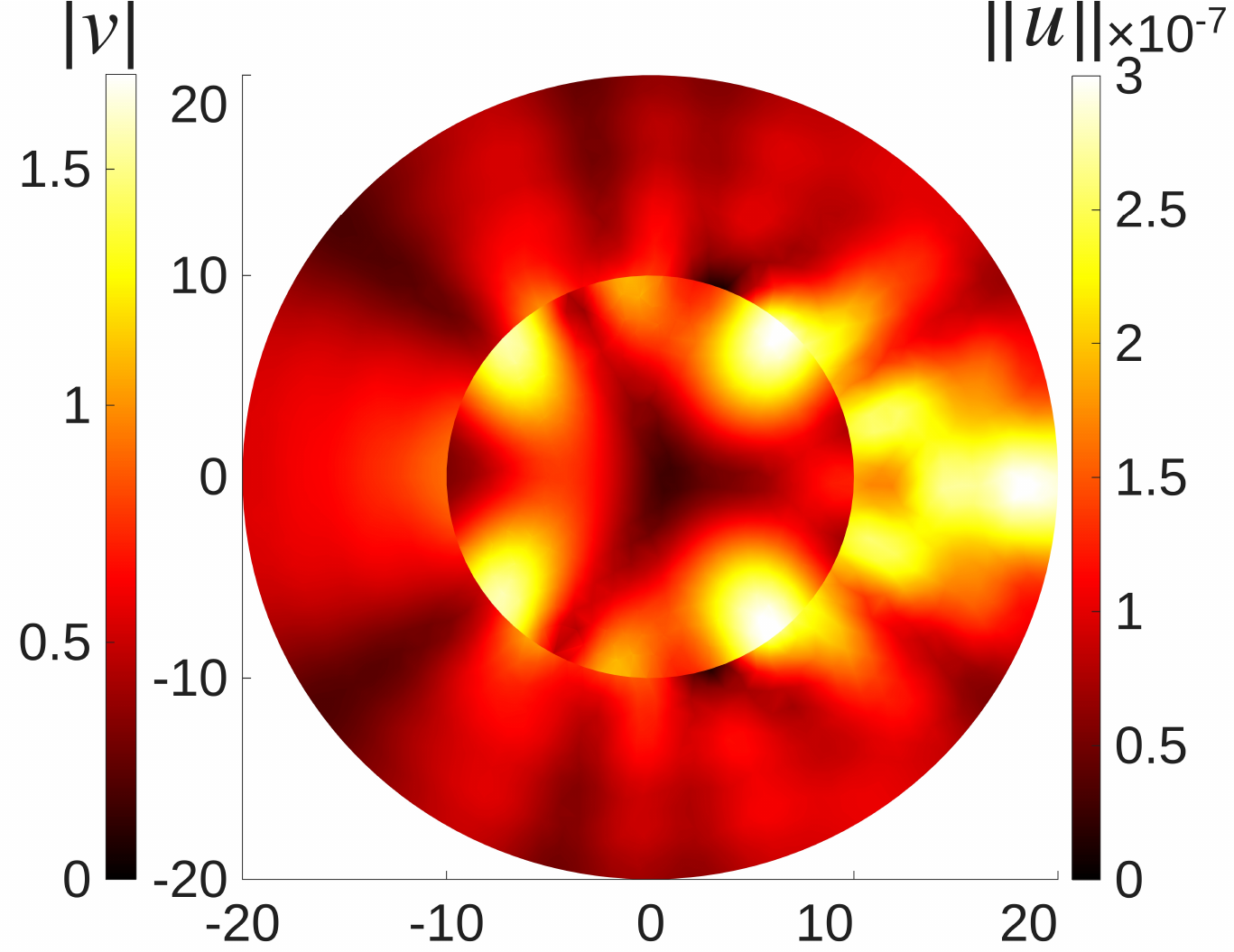} \qquad &
\includegraphics[width= 0.3\linewidth]{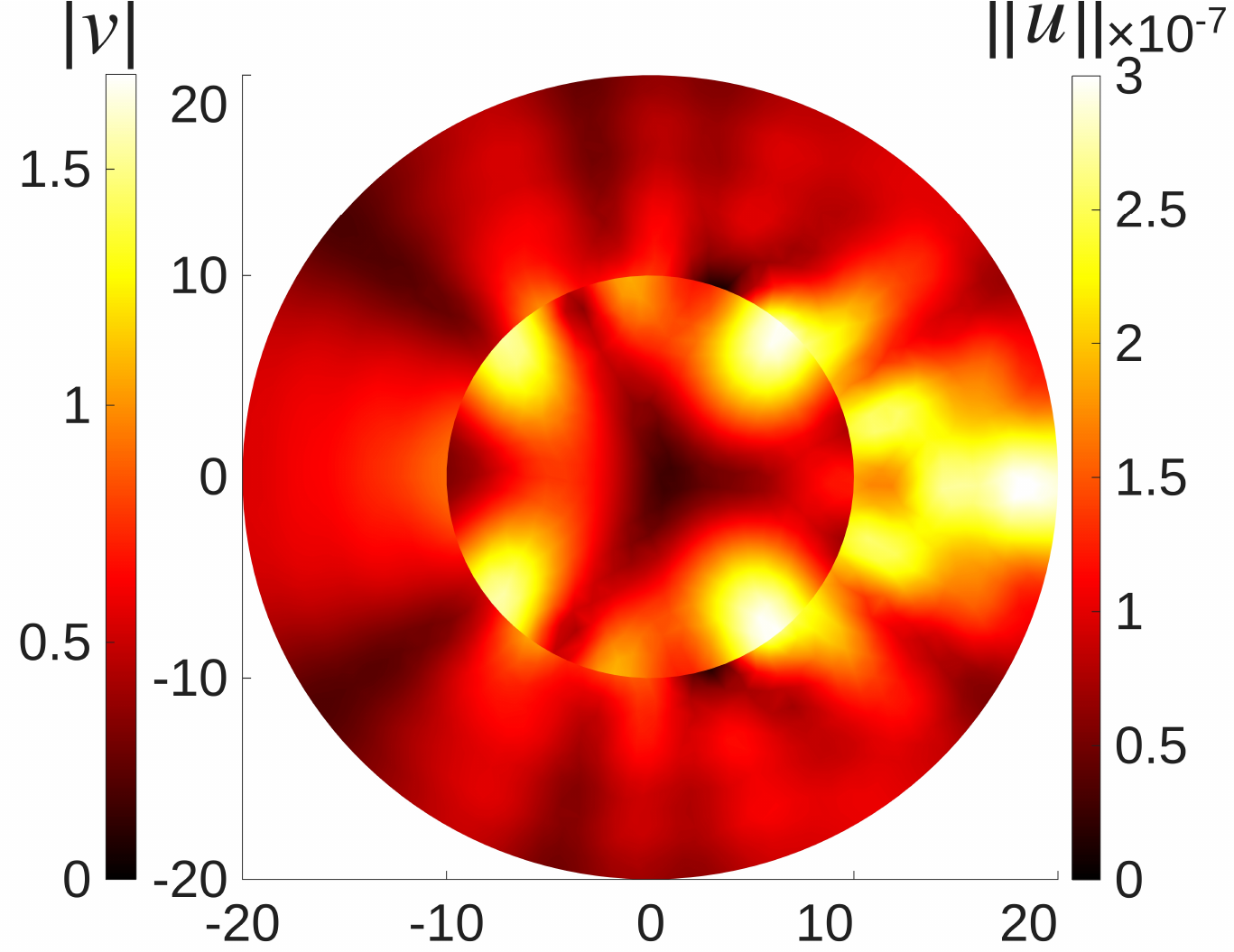}
\end{tabular}
\caption{HDG approximations to the solutions of \eqref{model-problem}. The top row uses level mesh 1 while the bottom row mesh level 2 as depicted in Figure \ref{fig:meshAndsolution}. Polynomial degrees $k=1,2,3$ vary over columns. The colorbars on the left depict magnitude of velocity potential, while the ones on the right denote magnitude of elastic displacement. Both are measured in millimeters--note the vast difference in magnitude.} 
\label{fig:plots_coupled2_mesh0}
\end{figure}
%
\section{Concluding remarks}
The current work presents what---to the authors's best knowledge---is the first analysis and proof of convergence of an HDG discretization for the a Laplace--domain system modeling the interaction between acoustic and elastic waves on a bounded domain. The numerical experiments suggest that convergence rates superior to those theoretically obtained can be expected. The challenge of rigorously establishing such improved rates remains outstanding.

In practical applications, it is often the case that the domain of interest is unbounded. The work presented here is a first step towards a discretization of such physically meaningful cases. In particular, the treatment of the coupling between the scheme analyzed in this communication with a boundary integral formulation for the exterior problem is the subject of ongoing work.

%
\section*{Acknowledgments}
%
Fernando Artaza-Covarrubias was partially funded by ANID-Chile through the grant Fondecyt Regular 1240183.  Tonatiuh S\'anchez-Vizuet was partially funded by the U. S. National Science Foundation through the grant NSF-DMS-2137305. Manuel Solano was partially funded by ANID-Chile through the grants Fondecyt Regular 1240183 and Basal FB210005. 

\newpage
\bibliographystyle{abbrv}
\bibliography{references}

\begin{thebibliography}{10}

\bibitem{AlAn:2024}
C.~J.~S. Alves and P.~R.~S. Antunes.
\newblock Wave scattering problems in exterior domains with the method of fundamental solutions.
\newblock {\em Numerische Mathematik}, 156(2):375–394, Feb. 2024.

\bibitem{ArRoVe:2019}
R.~Araya, R.~Rodríguez, and P.~Venegas.
\newblock Numerical analysis of a time domain elastoacoustic problem.
\newblock {\em IMA Journal of Numerical Analysis}, 40(2):1122--1153, 01 2019.

\bibitem{BaDjEs014}
H.~Barucq, R.~Djellouli, and E.~Estecahandy.
\newblock Efficient dg‐like formulation equipped with curved boundary edges for solving elasto‐acoustic scattering problems.
\newblock {\em International Journal for Numerical Methods in Engineering}, 98(10):747–780, Mar. 2014.

\bibitem{BeMaSo2025}
I.~Berm\'udez, J.~Manr\'iquez, and M.~Solano.
\newblock A hybridizable discontinuous {G}alerkin method for {S}tokes/{D}arcy coupling on dissimilar meshes.
\newblock {\em IMA J. Numer. Anal.}, 2025.

\bibitem{BrSaSa:2016}
T.~S. Brown, T.~S\'anchez-Vizuet, and F.-J. Sayas.
\newblock Evolution of a semidiscrete system modeling the scattering of acoustic waves by a piezoelectric solid.
\newblock {\em ESAIM: Mathematical Modelling and Numerical Analysis (M2AN)}, 52(2):423–455, Mar. 2018.

\bibitem{BuLoOs2018}
R.~Bustinza, B.~L\'opez-Rodr\'iguez, and M.~Osorio.
\newblock An a priori error analysis of an {HDG} method for an eddy current problem.
\newblock {\em Mathematical Methods in the Applied Sciences}, 41(7):2795--2810, 2018.

\bibitem{CaLoOsSo2020}
L.~Camargo, B.~López-Rodríguez, M.~Osorio, and M.~Solano.
\newblock An {HDG} method for {M}axwell’s equations in heterogeneous media.
\newblock {\em Computer Methods in Applied Mechanics and Engineering}, 368:113178, 2020.

\bibitem{CaSo:2023}
J.~M. C\'ardenas and M.~Solano.
\newblock {A high order unfitted hybridizable discontinuous {G}alerkin method for linear elasticity}.
\newblock {\em IMA Journal of Numerical Analysis}, 44(2):945--979, 05 2023.

\bibitem{CeCoNgPe2013}
A.~Cesmelioglu, B.~Cockburn, N.~C. Nguyen, and J.~Peraire.
\newblock Analysis of {HDG} methods for {O}seen equations.
\newblock {\em J. Sci. Comput.}, 55(2):392--431, May 2013.

\bibitem{CeCoQi2017}
A.~Cesmelioglu, B.~Cockburn, and W.~Qiu.
\newblock Analysis of a hybridizable discontinuous {G}alerkin method for the steady-state incompressible {N}avier--{S}tokes equations.
\newblock {\em Math. Comp.}, 86(306):1643--1670, 2017.

\bibitem{ChPeXu2013}
H.~Chen, P.~Lu, and X.~Xu.
\newblock A hybridizable discontinuous {G}alerkin method for the {H}elmholtz equation with high wave number.
\newblock {\em SIAM Journal on Numerical Analysis}, 51(4):2166--2188, 2013.

\bibitem{ChQiSh2018}
H.~Chen, W.~Qiu, and K.~Shi.
\newblock A priori and computable a posteriori error estimates for an {HDG} method for the coercive {M}axwell equations.
\newblock {\em Computer Methods in Applied Mechanics and Engineering}, 333:287 -- 310, 2018.

\bibitem{ChQiShSo2017}
H.~Chen, W.~Qiu, K.~Shi, and M.~Solano.
\newblock A superconvergent {HDG} method for the {M}axwell equations.
\newblock {\em Journal of Scientific Computing}, 70(3):1010--1029, 2017.

\bibitem{Chen2012}
Y.~Chen and B.~Cockburn.
\newblock Analysis of variable-degree {HDG} methods for convection-diffusion equations. {P}art {I}: general nonconforming meshes.
\newblock {\em {IMA} Journal of Numerical Analysis}, 32(4):1267--1293, Feb. 2012.

\bibitem{Chen2013}
Y.~Chen and B.~Cockburn.
\newblock Analysis of variable-degree {HDG} methods for convection-diffusion equations. part {II}: Semimatching nonconforming meshes.
\newblock {\em Mathematics of Computation}, 83(285):87--111, May 2013.

\bibitem{CoGoGu:2009}
B.~Cockburn, J.~Gopalakrishnan, and J.~Guzmán.
\newblock A new elasticity element made for enforcing weak stress symmetry.
\newblock {\em Mathematics of Computation}, 79(271):1331--1349, 2010.

\bibitem{CoGoLa2009}
B.~Cockburn, J.~Gopalakrishnan, and R.~Lazarov.
\newblock Unified hybridization of discontinuous {G}alerkin, mixed, and continuous {G}alerkin methods for second order elliptic problems.
\newblock {\em SIAM J. Numer. Anal.}, 47(2):1319--1365, 2009.

\bibitem{CoGoNgPe2011}
B.~Cockburn, J.~Gopalakrishnan, N.~C. Nguyen, J.~Peraire, and F.-J. Sayas.
\newblock Analysis of {HDG} methods for {S}tokes flow.
\newblock {\em Math. Comp.}, 80:723--760, 2011.

\bibitem{CoGoSa2010}
B.~Cockburn, J.~Gopalakrishnan, and F.-J. Sayas.
\newblock A projection-based error analysis of {HDG} methods.
\newblock {\em Mathematics of Computation}, 79(271):1351–1367, Mar. 2010.

\bibitem{CoQiSo2014}
B.~Cockburn, W.~Qiu, and M.~Solano.
\newblock A priori error analysis for {HDG} methods using extensions from subdomains to achieve boundary conformity.
\newblock {\em Mathematics of Computation}, 83, 03 2014.

\bibitem{Ke}
B.~Cockburn and K.~Shi.
\newblock {Superconvergent {HDG} methods for linear elasticity with weakly symmetric stresses}.
\newblock {\em IMA Journal of Numerical Analysis}, 33(3):747--770, 10 2012.

\bibitem{CoSo2012}
B.~Cockburn and M.~Solano.
\newblock Solving {D}irichlet boundary-value problems on curved domains by extensions from subdomains.
\newblock {\em SIAM Journal on Scientific Computing}, 34(1):A497--A519, 2012.

\bibitem{ErnDiPietro2012}
D.~A. Di~Pietro and A.~Ern.
\newblock {\em Mathematical Aspects of Discontinuous {G}alerkin Methods}.
\newblock Springer Berlin Heidelberg, 2012.

\bibitem{DoGaMe:2015}
C.~Domínguez, G.~N. Gatica, and S.~Meddahi.
\newblock A posteriori error analysis of a fully-mixed finite element method for a two-dimensional fluid-solid interaction problem.
\newblock {\em Journal of Computational Mathematics}, 33(6):606--641, 2015.

\bibitem{EnMa:1977}
B.~Engquist and A.~Majda.
\newblock Absorbing boundary conditions for the numerical simulation of waves.
\newblock {\em Mathematics of Computation}, 31(139):629--651, 1977.

\bibitem{FePeXu2016}
X.~Feng, P.~Lu, and X.~Xu.
\newblock A hybridizable discontinuous {G}alerkin method for the time-harmonic {M}axwell equations with high wave number.
\newblock {\em Computational Methods in Applied Mathematics}, 16(3):429--445, 2016.

\bibitem{FuCoSt2014}
G.~Fu, B.~Cockburn, and H.~Stolarski.
\newblock Analysis of an {HDG} method for linear elasticity.
\newblock {\em International Journal for Numerical Methods in Engineering}, 102(3-4):551–575, Oct. 2014.

\bibitem{FuJiQi2016}
G.~{Fu}, Y.~{Jin}, and W.~{Qiu}.
\newblock {Parameter-free superconvergent $H(\mathrm{div})$-conforming {HDG} methods for the {B}rinkman equations}.
\newblock {\em IMA Journal of Numerical Analysis}, 39(2):957--982, 03 2018.

\bibitem{FuQiZh2015}
G.~Fu, W.~Qiu, and W.~Zhang.
\newblock An analysis of {HDG} methods for convection--dominated diffusion problems.
\newblock {\em ESAIM: M2AN}, 49(1):225--256, 2015.

\bibitem{GaHeMe:2014}
G.~N. Gatica, N.~Heuer, and S.~Meddahi.
\newblock Coupling of mixed finite element and stabilized boundary element methods for a fluid–solid interaction problem in 3d.
\newblock {\em Numerical Methods for Partial Differential Equations}, 30(4):1211--1233, 2014.

\bibitem{GaMaMe:2012}
G.~N. Gatica, A.~Márquez, and S.~Meddahi.
\newblock Analysis of the coupling of {L}agrange and {A}rnold-{F}alk-{W}inther finite elements for a fluid-solid interaction problem in three dimensions.
\newblock {\em SIAM Journal on Numerical Analysis}, 50(3):1648--1674, 2012.

\bibitem{GaSe2016}
G.~N. Gatica and F.~A. Sequeira.
\newblock A priori and a posteriori error analyses of an augmented {HDG} method for a class of quasi-{N}ewtonian {S}tokes flows.
\newblock {\em J. Sci. Comput.}, 69:1192--1250, 2016.

\bibitem{GaSe2017}
G.~N. Gatica and F.~A. Sequeira.
\newblock Analysis of the {HDG} method for the {S}tokes--{D}arcy coupling.
\newblock {\em Numer. Methods Partial Differential Equations}, 33(3):885--917, 2017.

\bibitem{GrMo2011}
R.~Griesmaier and P.~Monk.
\newblock Error analysis for a hybridizable discontinuous {G}alerkin method for the {H}elmholtz equation.
\newblock {\em Journal of Scientific Computing}, 49(3):291--310, 2011.

\bibitem{Gu:2010}
J.~Guzmán.
\newblock {A unified analysis of several mixed methods for elasticity with weak stress symmetry}.
\newblock {\em Journal of Scientific Computing}, 44:156--169, 2010.

\bibitem{HaTr:1988}
L.~Halpern and L.~N. Trefethen.
\newblock {Wide‐angle one‐way wave equations}.
\newblock {\em The Journal of the Acoustical Society of America}, 84(4):1397--1404, 10 1988.

\bibitem{Hi:1987}
R.~L. Higdon.
\newblock Numerical absorbing boundary conditions for the wave equation.
\newblock {\em Mathematics of Computation}, 49(179):65--90, 1987.

\bibitem{HsSa:2021}
G.~C. Hsiao and T.~S\'anchez{-}Vizuet.
\newblock Time-dependent wave-structure interaction revisited: Thermo-piezoelectric scatterers.
\newblock {\em Fluids}, 6(3), 2021.
\newblock (arXiv: 2102.04118).

\bibitem{HsSaSa:2016}
G.~C. Hsiao, T.~S\'anchez-Vizuet, and F.-J. Sayas.
\newblock Boundary and coupled boundary{-}finite element methods for transient wave{-}structure interaction.
\newblock {\em IMA Journal of Numerical Analysis}, 37(1):237--265, 2016.
\newblock (arXiv:1509.01713).

\bibitem{HuKaMo2008}
T.~Huttunen, J.~Kaipio, and P.~Monk.
\newblock An ultra-weak method for acoustic fluid–solid interaction.
\newblock {\em Journal of Computational and Applied Mathematics}, 213(1):166–185, Mar. 2008.

\bibitem{Huynh2013}
L.~N.~T. Huynh, N.~C. Nguyen, J.~Peraire, and B.~C. Khoo.
\newblock {A high-order hybridizable discontinuous {G}alerkin method for elliptic interface problems}.
\newblock {\em International Journal for Numerical Methods in Engineering}, 93(2):183--200, jan 2013.

\bibitem{LiTa:1997}
Q.-H. Liu and J.~Tao.
\newblock {The perfectly matched layer for acoustic waves in absorptive media}.
\newblock {\em The Journal of the Acoustical Society of America}, 102(4):2072--2082, 10 1997.

\bibitem{MaNgSo2022}
J.~Manr\'iquez, N.-C. Nguyen, and M.~Solano.
\newblock A dissimilar non-matching {HDG} discretization for {S}tokes flows.
\newblock {\em Comput. Methods Appl. Mech. Engrg.}, 399:Paper No. 115292, 30, 2022.

\bibitem{MeMoRo:2014}
S.~Meddahi, D.~Mora, and R.~Rodríguez.
\newblock Finite element analysis for a pressure–stress formulation of a fluid–structure interaction spectral problem.
\newblock {\em Computers \& Mathematics with Applications}, 68(12, Part A):1733--1750, 2014.

\bibitem{Mottier2025}
R.~Mottier, A.~Ern, R.~Khot, and L.~Guillot.
\newblock Hybrid high-order methods for elasto-acoustic wave propagation in the time domain.
\newblock {\em ESAIM: Mathematical Modelling and Numerical Analysis}, 59(5):2685–2715, Sept. 2025.

\bibitem{NgPeCo2009}
N.~C. Nguyen, J.~Peraire, and B.~Cockburn.
\newblock An implicit high--order hybridizable discontinuous {G}alerkin method for linear convection--diffusion equations.
\newblock {\em J. Comput. Phys.}, 228(9):3232--3254, 2009.

\bibitem{NgPeCo2011}
N.~C. Nguyen, J.~Peraire, and B.~Cockburn.
\newblock Hybridizable discontinuous {G}alerkin methods for the time-harmonic {M}axwell's equations.
\newblock {\em Journal of Computational Physics}, 230(19):7151--7175, 2011.

\bibitem{NgPeCo2011NS}
N.~C. Nguyen, J.~Peraire, and B.~Cockburn.
\newblock An implicit high-order hybridizable discontinuous {G}alerkin method for the incompressible {N}avier--{S}tokes equations.
\newblock {\em J. Comput. Phys.}, 230(4):1147--1170, 2011.

\bibitem{QiShSh2018}
W.~Qiu, J.~Shen, and K.~Shi.
\newblock {HDG} method for linear elasticity with strong symmetric stresses.
\newblock {\em Mathematics of Computations}, 87:69--93, 2018.

\bibitem{SaSaSo:2021a}
N.~S{\'{a}}nchez, T.~S{\'{a}}nchez-Vizuet, and M.~Solano.
\newblock Error analysis of an unfitted {HDG} method for a class of non-linear elliptic problems.
\newblock {\em Journal of Scientific Computing}, 90(3), Feb. 2022.
\newblock (arXiv: 2105.03560).

\bibitem{SaSaSo:2019}
N.~S{\'{a}}nchez, T.~S{\'{a}}nchez-Vizuet, and M.~E. Solano.
\newblock A priori and a posteriori error analysis of an unfitted {HDG} method for semi-linear elliptic problems.
\newblock {\em Numerische Mathematik}, 148(4):919--958, Aug. 2021.
\newblock (arXiv:1911.12298).

\bibitem{SaSaSo2022}
N.~S\'anchez, T.~S\'anchez-Vizuet, and M.~E. Solano.
\newblock Afternote to {“Coupling at a Distance”}: Convergence analysis and a priori error estimates.
\newblock {\em Computational Methods in Applied Mathematics}, 22(4):945--970, 2022.

\bibitem{tesis_tonatiuh}
T.~S\'anchez-Vizuet.
\newblock {\em Integral and coupled integral-volume methods for transient problems in wave-structure interaction}.
\newblock Ph{D} thesis, University of Delaware, Newark, DE, 2016.

\bibitem{SV:2024}
T.~S\'anchez-Vizuet.
\newblock A symmetric boundary integral formulation for time–domain acoustic-elastic scattering.
\newblock {\em Computational Mechanics}, 77(3):695–703, Oct. 2025.
\newblock arXiv:2502.04767.

\bibitem{SaSo:2018}
T.~S\'anchez-Vizuet and M.~E. Solano.
\newblock A {H}ybridizable {D}iscontinuous {G}alerkin solver for the {G}rad-{S}hafranov equation.
\newblock {\em Computer Physics Communications}, 235:120 -- 132, 2019.
\newblock (arXiv:1712.04148).

\bibitem{SaSoCe:2018}
T.~S\'anchez{-}Vizuet, M.~E. Solano, and A.~J. Cerfon.
\newblock Adaptive hybridizable discontinuous {G}alerkin discretization of the {G}rad-{S}hafranov equation by extension from polygonal subdomains.
\newblock {\em Computer Physics Communications}, 255:107239, 2020.
\newblock (arXiv:1903.01724).

\bibitem{SoTeNgPe2022}
M.~Solano, S.~Terrana, N.-C. Nguyen, and J.~Peraire.
\newblock An {HDG} method for dissimilar meshes.
\newblock {\em IMA J. Numer. Anal.}, 42(2):1665--1699, 2022.

\bibitem{HaTr:1986}
L.~N. Trefethen and L.~Halpern.
\newblock Well-posedness of one-way wave equations and absorbing boundary conditions.
\newblock {\em Mathematics of Computation}, 47(176):421--435, 1986.

\bibitem{Wang2013HDG}
B.~{Wang} and B.~C. {Khoo}.
\newblock {Hybridizable discontinuous {G}alerkin method ({HDG}) for {S}tokes interface flow}.
\newblock {\em Journal of Computational Physics}, 247:262--278, Aug. 2013.

\bibitem{ZhWu2021}
B.~Zhu and H.~Wu.
\newblock Preasymptotic error analysis of the {HDG} method for {H}elmholtz equation with large wave number.
\newblock {\em Journal of Scientific Computing}, 87(2):1--34, 2021.

\end{thebibliography}

\appendix

%
\section{Convergence studies for the uncoupled problems}
%
The main contribution of the current article is the analysis of the discretization of the coupled transmission problem. For the sake of completeness, here we present numerical experiments that confirm the theoretical rates of convergence presented for the uncoupled acoustic and elastic discretizations.  
%
\subsection{Acoustic problem}
%
To test our HDG scheme applied to the acoustic problem, we consider equations \eqref{Equation1C}-\eqref{Equation1D} complemented with Dirichlet boundary conditions $v=g_D$ on $\partial \Omega_A$. We take a manufactured acoustic field $v(x,y)=\sin(x)\sin(y)$. The source $f$ and boundary data $g_D$ are set in such a way that $v$ satisfies \eqref{Equation1C}-\eqref{Equation1D} 
in a domain $\Omega_A=(0,1)^2$, with $c=1$ and, for example, $s=2-i$. The stabilization parameter $\tau_A$ is taken to be equal to one everywhere. As it can be inferred from \autoref{thm:error_estimates} and \eqref{bound:uhat} (see also \cite{CoGoSa2010}), the theoretical orders of convergence for this case are $h^{k+1}$ for $v$ and $\bm q$; and $h^{k+2}$ for the numerical trace, since the domain is convex ($s_a=1$).

We consider quasi-uniform refinements of $\Omega_A$ and set $k\in \{1,2,3\}$ in the local spaces. \autoref{fig:acustica} shows the results obtained for this problem, where $N$ is the number of mesh triangles. Note that for the errors in $\q$ and $v$ the optimal theoretical order of convergence $k+1$ was reached. In turn, for the numerical trace we can see an order of superconvergence $k+2$, as expected.

\begin{figure}[htp]
\centering
\begin{tabular}{ccc}
  \normalsize{Degree $k=1$} & \normalsize{Degree $k=2$} & \normalsize{Degree $k=3$} \\
  \includegraphics[width= 0.32\linewidth]{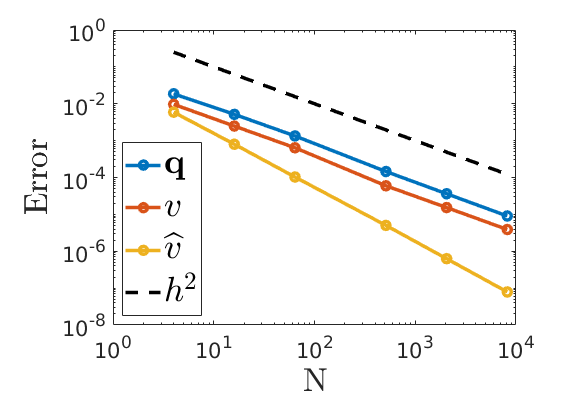} &
  \includegraphics[width= 0.32\linewidth]{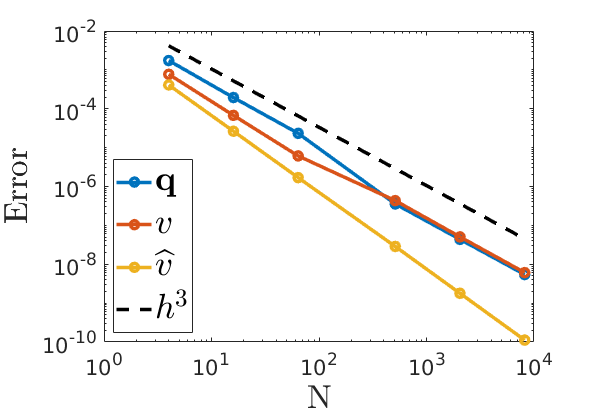} &
 \includegraphics[width= 0.32\linewidth]{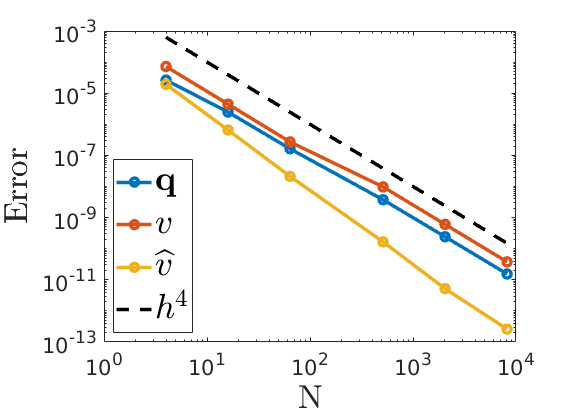}
\end{tabular}
\caption{Discretization error as a function of the number of triangles in the domain for the acoustic problem.}
\label{fig:acustica}
\end{figure}
%
\subsection{Elastic problem.}
%
Analogously to the previous subsection, let us apply the HDG scheme to the equations \eqref{Equation1A}-\eqref{Equation1B} considering $\ome=(0,1)^2, \rho_E=1, s=2-i$ and $\tau_E = 1$ everywhere. The source $\efe$ and the Dirichlet boundary condition are defined such that
\[
\boldsymbol{u}(x,y)=\begin{pmatrix}
    \sin(\pi x)\cos(\pi y)\\
    \cos(\pi x)\sin(\pi y)
\end{pmatrix}, \quad (x,y)\in (0,1)^2,
\]
is the exact solution of the problem.

It is known that the Lamé's first parameter ($\lambda$) and the shear modulus ($\mu$) (or Lamé's second parameter) satisfy the following expressions in terms of the Young's modulus ($E$) and the Poisson's ratio ($\nu$):
\[
\lambda = \dfrac{E\nu}{(1+\nu)(1-2\nu)} \qquad\text{and} \qquad \mu=\dfrac{E}{2(1+\nu)},
\]
so let us take $E=1$ and two values of $\nu$, $0.3$ and $0.49999$ (a nearly incompressible isotropic material deformed elastically at small strains would have a Poisson's ratio of exactly $0.5$). Note that, by construction, the source term for this manufactured test depends on the value of $\nu$ and will in fact diverge as $\nu\to1/2$, hence the error bound loosens for values $\nu\approx 1/2$. This can be observed in the convergence plots in Figure \ref{fig:plots_elastica}, where the error shifts uniformly upwards as $\nu\to1/2$ without affecting the rate of convergence. This, however, does not contradict HDG's immunity to locking (as proven in \cite{FuCoSt2014}) as the constant in the error bound remains independent of $\lambda$ $\mu$, the dependence falls entirely on the source.

From \autoref{thm:error_estimates}, we can deduce that the theoretical order of convergence is $h^{k+1}$ for the displacements and the Cauchy stress tensor. Now, the negative powers of $h$ in \eqref{bound:e_rho_c} are due to the term $\ipbe{\eug}{\ten{v}\nor}$ in the proof of Lemma \ref{lema:e_rho}. This term arises when coupling the elasticity and acoustic equations. Since in this example there is no coupling, the term $\ipbe{\eug}{\ten{v}\nor}$ disappears and we can obtain that
\begin{align*}
\|\eg^c\|_{\te} 
\lesssim \Theta(\si,\uv,\g,\q,v). \end{align*}
Therefore, the theory guarantees an order $h^{k+1}$ for the rotation, which agrees with the results in \cite{Ke}. The same reason led the suboptimal estimates in \eqref{bound:uhat}. Since in this example we are considering only the elasticity problem in a convex domain, we have regularity $s_e=1$ and \eqref{bound:uhat} can be improved:
\[
\vertiii{\eug}_{\partial\mathcal{T}_E}\lesssim h^{k+2}.
\]
Moreover, the HDG scheme is also optimal in the nearly incompressible case \cite{CaSo:2023,Ke}.

The numerical results are shown in \autoref{fig:plots_elastica}. Observe that the experimental orders of convergence of the errors in $\si,\uv$ and $\g$, $k+1$, coincide with the theoretical results. In addition, for the numerical trace of $\uv$ we also have a superconvergence of order $k+2$.

 \begin{figure}
 \begin{tabular}{ccc}
  &\normalsize $\boldsymbol{\nu=0.3}$ & \\[1ex]
   Degree $k=1$ & Degree $k=2$ & Degree $k=3$ \\
    \includegraphics[width= 0.315\linewidth]{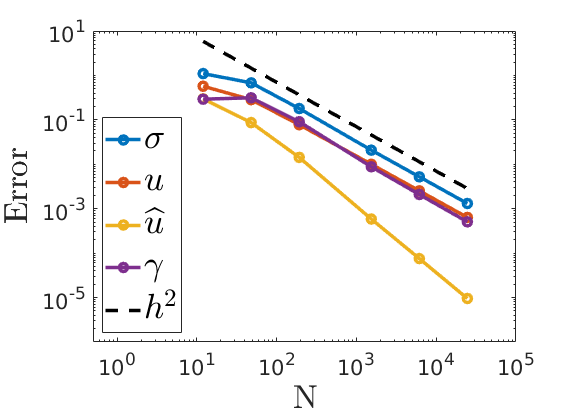} &  \includegraphics[width= 0.315\linewidth]{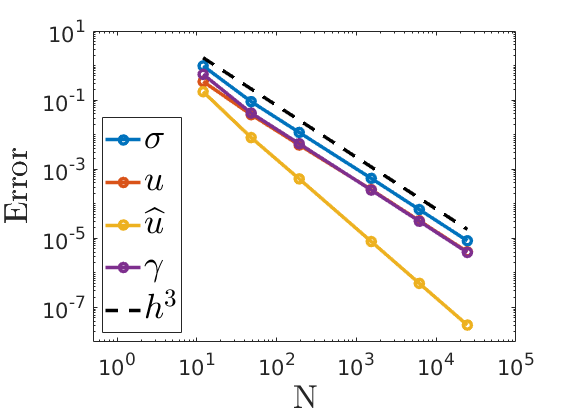} & \includegraphics[width= 0.315\linewidth]{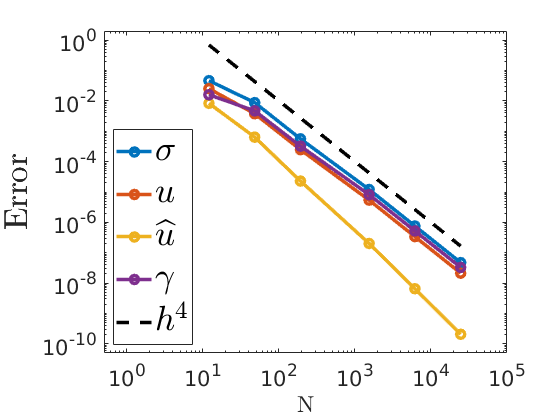} \\
  & \normalsize $\boldsymbol{\nu=0.49999}$ & \\[1ex]
   Degree $k=1$ & Degree $k=2$ & Degree $k=3$ \\
  \includegraphics[width= 0.315\linewidth]{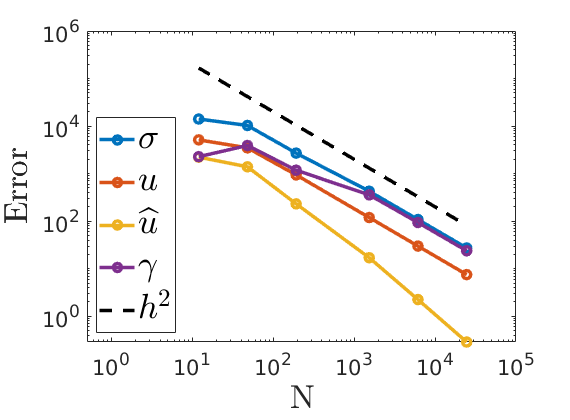} &  \includegraphics[width= 0.315\linewidth]{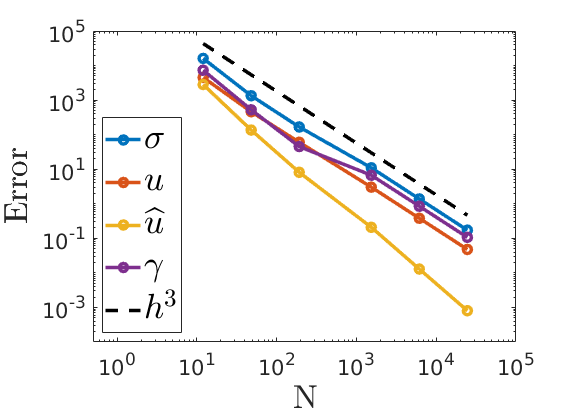} & \includegraphics[width= 0.315\linewidth]{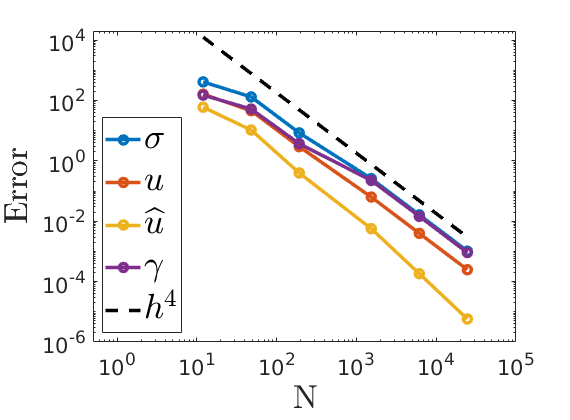}
 \end{tabular}
 \caption{Discretization error as a function of the number of elements in the elastic domain for Poisson's ratio $\nu =0.3$ (first row) and $\nu=0.49999$ (second row).}\label{fig:plots_elastica}
 \end{figure}

\end{document}